\documentclass[12pt]{article}
\usepackage{amssymb, amsthm}
\usepackage{amsmath}

\newtheorem{Rem}{Remark}
\newtheorem{sats}{Theorem}
\newtheorem{prop}{Proposition}
\newtheorem{lem}{Lemma}
\newtheorem{kor}{Corollary}
\newcommand{\banm}{\begin{anm}}
\newcommand{\eanm}{\end{anm}}

\newcommand{\loc}{{\rm loc}}

\begin{document}

\title
{The Dirichlet problem \\
for  non-divergence parabolic equations\\
with discontinuous in time coefficients\\ in a wedge
}
\author
{Vladimir Kozlov, Alexander Nazarov}

\date{}
\maketitle
\begin{abstract} We consider the Dirichlet problem in a wedge for parabolic
equation whose coefficients are measurable function of $t$. We
obtain coercive estimates in weighted $L_{p,q}$-spaces. The concept
of ``critical exponent'' introduced in the paper plays here the
crucial role. Various important properties of the critical exponent
are proved. We give applications to the Dirichlet problem for linear
and quasi-linear non-divergence parabolic equations with
discontinuous in time coefficients in cylinders $\Omega\times
(0,T)$, where $\Omega$ is a bounded domain with an edge or with a
conical point.
\end{abstract}

\section{Introduction}

Consider the parabolic equation
\begin{equation}\label{Jan1}
{\cal L}u=\partial_tu(x;t)-A_{ij}(t)D_iD_ju(x;t)=f(x;t).
\end{equation}
Here and elsewhere $D_i$ denotes the operator of differentiation with respect to $x_i$;
in particular, $Du=(D_1u, \dots, D_nu)$ is the gradient of $u$. By $\partial_tu$ we denote
the derivative of $u$ with respect to $t$.

The only assumptions
about the coefficients in (\ref{Jan1}) is that $A_{ij}$ are real
valued measurable functions of $t$ such that\footnote{We may suppose
also $A_{ij}\equiv A_{ji}$ without loss of generality.}
\begin{equation}\label{0.3}
\nu|\xi|^2\le A_{ij}(t)\xi_i\xi_j\le \nu^{-1}|\xi|^2, \qquad
\xi\in{\mathbb R}^n, \quad \nu=const>0.
\end{equation}

In \cite{Kr2}, \cite{Kr} it was shown by Krylov that for coercive
estimates of $\partial_tu$ and $D(Du)$ one needs no smoothness
assumptions on coefficients $A_{ij}$ with respect to $t$. The only
assumption which is needed is estimate (\ref{0.3}). Solvability
results for the whole space $\mathbb R^n\times\mathbb R$ for
equation (\ref{Jan1}) in $L_{p,q}$ spaces were proved in \cite{Kr2};
solvability of the Dirichlet problem in the half-space $\mathbb
R^n_+\times\mathbb R$ was established in weighted $L_{p,q}$ spaces
by Krylov \cite{Kr} for particular range of weights and by the
authors \cite{KN} for the whole range of weights.\medskip

This paper addresses solvability results for the Dirichlet problem
to (\ref{Jan1}) in the wedge. Namely, let $K$ be a cone in $\mathbb
R^m$, $2\le m\le n$, and let ${\cal K}=K\times\mathbb R^{n-m}$. We
assume that $\omega=K\cap {\mathbb S}^{m-1}$ is of class ${\cal
C}^{1,1}$. We underline that the case $m=n$ where ${\cal K}=K$ is
not excluded. We are looking for solutions to the following
Dirichlet problem
\begin{equation}\label{Jan1a}
{\cal L}u=f\quad\mbox{in}\quad {\cal K}\times \mathbb R;\qquad
u=0\quad\mbox{on}\quad \partial {\cal K}\times \mathbb R.
\end{equation}

When the coefficients are independent on $t$, solvability results in weighted
$L_2$ spaces and even asymptotic decomposition for solutions of problem
(\ref{Jan1a}) in the case of the cone ($n=m$) were obtained in \cite{KM1}-\cite{K0}.
In the case of the dihedral angle ($m=2$) weighted $L_p$-coercive estimates were proved
in \cite{Sol} (see also \cite{Sol1} where solvability in H\"older classes was
established); for arbitrary wedge these estimates and corresponding estimates in
anisotropic spaces were established in \cite{Na}. In the constant coefficient case
by using a change of variables the elliptic part of the operator in (\ref{Jan1})
can be reduced to the Laplacian, for which the interval of the admissible weights
in the coercive estimates is described in terms of the first positive eigenvalue
$\lambda_{\cal D}$ of a quadratic operator pencil associated with the Beltrami-Laplacian
on $\omega$, see \cite{K0}, \cite{Na}. If coefficients depend on $t$ the above reduction
is impossible. We define the critical exponents $\lambda_c^{\pm}$ of ${\cal L}$ in ${\cal K}$,
which can be considered as a generalization of $\lambda_{\cal D}$ onto the case when coefficients
depend on $t$. This critical exponents coincide with $\lambda_{\cal D}$ when
$A_{ij}(t)\equiv\delta_{ij}$, i.e. (\ref{Jan1}) is the heat equation. Other important properties of
$\lambda_c^{\pm}$ are established in Theorem \ref{lambda}.

 In order to formulate our main result let us introduce two classes of anisotropic spaces.
For $1<p,q<\infty$ we define $L_{p,q}=L_{p,q}({\cal K}\times \mathbb R)$
as a space of functions with the finite norm
$$
\|f\|_{p,q}=\big\|\|f(\cdot;t)\|_{p,{\cal K}}\big\|_{q,\mathbb R}=
\Big (\int\limits_{\mathbb R}\Big (\int\limits_{\cal K}|f(x;t)|^pdx\Big )^{q/p}dt\Big )^{1/q}.
$$
Similarly, the space $\widetilde L_{p,q}=\widetilde L_{p,q}({\cal K}\times \mathbb R)$
consists of functions with the finite norm
$$|\!|\!|f|\!|\!|_{p,q}=\big\|\|f(x,\cdot)\|_{q,\mathbb R}\big\|_{p,{\cal K}}=
\Big (\int\limits_{\cal K}\Big (\int\limits_{\mathbb R}|f(x;t)|^qdt\Big )^{p/q}dx\Big )^{1/p}.
$$

We denote by $W^{2,1}_{p,q,(\mu)}$ and $\widetilde {W}^{2,1}_{p,q,(\mu)}$ the sets of functions with the
finite norms
$$
\|u \|_{W^{2,1}_{p,q,(\mu)}}=\||x'|^\mu\partial_t u\|_{p,q}+\||x'|^\mu D(Du)\|_{p,q}+
\||x'|^{\mu-2} u\|_{p,q}
$$
and
$$
\|u \|_{\widetilde{W}^{2,1}_{p,q,(\mu)}}=|\!|\!||x'|^\mu\partial_t u|\!|\!|_{p,q}+
|\!|\!||x'|^\mu D(Du)|\!|\!|_{p,q}+ |\!|\!| |x'|^{\mu-2} u|\!|\!|_{p,q}
$$
respectively.\medskip

\begin{sats}\label{main}
Let $\lambda_c^{\pm}$ be the critical exponents defined in {\rm Sect.\ref{Rn}}. Suppose that
\begin{equation}\label{mu}
2-\frac mp-\lambda_c^+<\mu<m-\frac mp+\lambda_c^-.
\end{equation}
Then for any $f\in L_{p,q}$ (respectively, for $f\in \widetilde L_{p,q}$)
there is a solution of the boundary value problem {\rm (\ref{Jan1a})} satisfying the following estimates:
\begin{equation}\label{beztilde}
\|u \|_{W^{2,1}_{p,q,(\mu)}}\le C\ \||x'|^\mu f\|_{p,q};
\end{equation}
\begin{equation}\label{tilde}
\|u \|_{\widetilde{W}^{2,1}_{p,q,(\mu)}}\le C\ |\!|\!||x'|^\mu f|\!|\!|_{p,q}.
\end{equation}
This solution is unique in the space $W^{2,1}_{p,q,(\mu)}$ (respectively, $\widetilde {W}^{2,1}_{p,q,(\mu)}$).
\end{sats}

Our paper is organized as follows. The critical exponents
$\lambda_c^{\pm}$ of the operator ${\cal L}$ in ${\cal K}$ are
defined in Section \ref{Rn}, where we present also their various
properties. In Section \ref{Greenest} we estimate the Green function
of the BVP (\ref{Jan1a}) and its derivatives. Theorem \ref{main} is
proved in Section \ref{Rn+2}. In Section \ref{solv} we give some
applications of this result to the Dirichlet problem for linear and
quasi-linear non-divergence parabolic equations with discontinuous
in time coefficients in cylinders $\Omega\times (0,T)$, where
$\Omega$ is a bounded domain with an edge or with a conical point.
Auxiliary estimates are collected in the Appendix.\medskip

Let us introduce some notation: $x=(x',x'')=(x_1,\ldots,x_n)$ is a point in $\mathbb R^n$
(here $x'=(x_1,\ldots,x_m)\in\mathbb R^n$ and $x''=(x_{m+1},\ldots,x_n)\in \mathbb R^{n-m}$);
$(x;t)$ is a point in $\mathbb R^{n+1}$.

If $\Omega$ is a domain in $\mathbb R^n$ then $\partial\Omega$ stands for its boundary.
For a cylinder $Q=\Omega\times\,(t_1,t_2)$, we denote by
$\partial'Q=\big(\partial\Omega\times\,(t_1,t_2)\big)\cup\big(\overline{\Omega} \times \{t_1\}\big)$
its parabolic boundary.

For $x\in{\cal K}$, $d(x)$ is the distance from $x$ to $\partial {\cal K}$, and
$r_x=\frac{d(x)}{|x'|}$.

For $0<\theta<\pi$, $K^\theta=\{x'\in\mathbb R^m\,:\,\widehat{x',x_1}<\theta\}$ is a circular cone.
$$\begin{array}{ll}
B_R^n(x_0)=\{x\in\mathbb R^n\,:\,|x-x_0|<R\}; & B_R^n=B_R^n(0);\\
Q_R(x_0;t_0)=B_R^n(x_0)\times\,(t_0-R^2,t_0]; &  Q_R=Q_R(0;0);\\
{\cal B}_{\rho,R}^K(x_0)=\big(B_{\rho}^m(x'_0)\cap K\big)\times B_R^{n-m}(x''_0); &
{\cal B}_{\rho,R}^K={\cal B}_{\rho,R}^K(0);\\
{\cal Q}_{\rho,R}^K(x_0;t_0)={\cal B}_{\rho,R}^K(x_0)\times\,(t_0-R^2,t_0]; &
{\cal Q}_{\rho,R}^K={\cal Q}_{\rho,R}^K(0;0);\\
{\cal B}_R^K(x_0)={\cal B}_{R,R}^K(x_0); & {\cal B}_R^K={\cal B}_R^K(0);\\
{\cal Q}_R^K(x_0;t_0)={\cal Q}_{R,R}^K(x_0;t_0); & {\cal Q}_R^K={\cal Q}_R^K(0;0).
\end{array}
$$

The indeces $i, j$  vary from $1$ to $n$ while the indeces $k,\ell$ vary from $1$ to $m$.
Repeated indeces indicate summation.

By ${\cal V}(Q_R(x_0;t_0))$ we denote the set of functions $u$ with finite norm
\begin{equation*}
\|u\|_{{\cal V}(Q_R(x_0;t_0))}=\sup_{\tau\in\,]t_0-R^2,t_0]}\|u(\tau,\cdot)\|_{L_2(B_R(x_0))}
+\|Du\|_{L_2(Q_R(x_0;t_0)))}.
\end{equation*}
We write $u\in {\cal V}_\loc(Q_R(x_0;t_0))$ if
$u\in {\cal V}(Q_{R'}(x_0;t_0))$ for all $R'\in (0,R)$. Similarly we define spaces
${\cal V}({\cal Q}_R^K(x_0;t_0))$ and ${\cal V}_\loc({\cal Q}_R^K(x_0;t_0))$.

We use the letter $C$ to denote various positive
constants. To indicate that $C$ depends on some parameter $a$, we
sometimes write $C(a)$.

\section{Critical exponent}\label{Rn}

We define the critical exponent for the operator ${\cal L}$ and
the wedge ${\cal K}$ as the supremum of all $\lambda$
such that
\begin{equation}\label{0.4}
|u(x;t)|\leq C(\lambda,\kappa) \Big (\frac{|x'|}{R}\Big )^\lambda\sup_{
{\cal Q}_{\kappa R}^K(0;t_0)}|u|\quad\mbox{for}\quad (x;t)\in {\cal
Q}_{R/2}^K(0;t_0)
\end{equation}
for a certain $\kappa\in (1/2,1)$ independent of $t_0$, $R$ and $u$.
This inequality must be satisfied for all  $t_0\in\mathbb R$, $R>0$ and
for all  $u\in {\cal V}_\loc({\cal Q}_{R}^K(0;t_0))$  subject to
\begin{equation}\label{0.5}
{\cal L}u=0\quad\mbox{in}\quad {\cal Q}_R^K(0;t_0);\qquad
u\big|_{x\in\partial{\cal K}}=0.
\end{equation}
 We denote this critical exponent by $\lambda_c^+\equiv\lambda_c^+({\cal K},{\cal L})$.
Since $\lambda=0$  satisfies (\ref{0.4}) we conclude that $\lambda_c\geq0$.

We also consider operator
\begin{equation}\label{0.51}
\widehat{{\cal L}}=\partial_t-A_{ij}(-t)D_iD_j.
\end{equation}
The elliptic part of this operator evidently satisfies (\ref{0.3}). We introduce the
critical exponent for the operator $\widehat{{\cal L}}$ and the wedge ${\cal K}$ and define
$\lambda_c^-\equiv\lambda_c^-({\cal K},{\cal L})=\lambda_c^+({\cal K},\widehat{{\cal L}})$.
Note that all the properties below will be proved simultaneously for ${\lambda}_c^{\pm}$.

\begin{Rem}\label{init1}
Consider the initial-boundary value problem
\begin{equation}\label{init}
{\cal L}u=f\quad\mbox{in}\quad {\cal K}\times \mathbb R_+;\qquad
u=0\quad\mbox{on}\quad (\partial {\cal K}\times \mathbb R_+)\cup({\cal K}\times \mathbb R_-),
\end{equation}
where coefficients $A_{ij}(t)$ are given only for $t>0$. In this
case we should slightly change the definition of $\lambda^+_c({\cal
K},{\cal L})$ by assuming additionally $t_0-R^2>0$ in {\rm
(\ref{0.4})}. Correspondingly, in the definition of
$\lambda^+_c({\cal K},\widehat {\cal L})$ we assume additionally
$t_0<0$ in {\rm (\ref{0.4})}.
\end{Rem}

\begin{lem}\label{netkappa}
The above definition of $\lambda_c^{\pm}({\cal K},{\cal L})$ does not depend on $\kappa\in (1/2,1)$.
\end{lem}

\begin{proof}
We have to show that any solution to (\ref{0.5}) satisfying (\ref{0.4}) with some $\lambda$ and $\kappa=\kappa_1$
satisfies this estimate with the same $\lambda$ and $\kappa=\kappa_2$. This statement is trivial if
 $\kappa_1\le\kappa_2$, so we assume that $1/2<\kappa_2<\kappa_1<1$.

Define $\rho=\frac {2\kappa_2-1}{2\kappa_1-1}R$ and consider the set $\mathfrak Q$ of points $(0,x_0'';\tau_0)$
such that
${\cal Q}^K_{\kappa_1\rho}(0,x_0'';\tau_0)\subset{\cal Q}^K_{\kappa_2R}(0;t_0)$. Then direct calculation shows that
$${\cal Q}^K_{\frac \rho2,\frac R2}(0;t_0)\subset\bigcup\limits_{(0,x_0'';t_0)\in\mathfrak Q}
{\cal Q}^K_{\frac \rho2}(0,x_0'';\tau_0).
$$
Applying (\ref{0.4}) with $\kappa=\kappa_1$ and $R=\rho$ in ${\cal Q}^K_{\frac \rho2}(0,x_0'';\tau_0)$ for
$(0,x_0'';\tau_0)\in\mathfrak Q$, we obtain
$$|u(x;t)|\leq C(\lambda,\kappa_1) \Big (\frac{|x'|}{\rho}\Big )^\lambda
\sup_{{\cal Q}_{\kappa_2 R}^K(0;t_0)}|u|\quad\mbox{for}\quad (x;t)\in {\cal Q}^K_{\frac \rho2,\frac R2}(0;t_0).
$$
This implies (\ref{0.4}) with $\kappa=\kappa_2$ and
$C(\lambda,\kappa_2)=\big(\frac {2\kappa_1-1}{2\kappa_2-1}\big)^{\lambda}\cdot
\max\{2^{\lambda},C(\lambda,\kappa_1)\}$.
\end{proof}

\begin{Rem}\label{Rem1} Instead of {\rm (\ref{0.4})} one can define the
critical exponent as the supremum of all $\lambda$ in the inequality
\begin{equation}\label{0.4a}
|u(x;t)|\leq \frac {C(\lambda,\kappa)}{R^{1+\frac n2}}\ \Big (\frac{|x'|}{R}\Big )^\lambda
\Big (\int\limits_{ {\cal Q}_{\kappa R}^K(0;t_0)}|u(x;t)|^2dxdt\Big )^{1/2}
\end{equation}
for $(x;t)\in {\cal Q}_{R/2}^K(0;t_0)$, where $\kappa\in (1/2,1)$
and the inequality must be valid for the same set of $t_0$, $R$ and
$u$ as above.  Clearly, {\rm (\ref{0.4})} follows from {\rm (\ref{0.4a})}.
In turn, {\rm (\ref{0.4})} implies {\rm (\ref{0.4a})} due to the local estimate
\begin{equation}\label{local}
\sup_{ {\cal Q}_{\kappa_1 R}^K(0;t_0)}|u|\leq
\frac {C(\kappa_1,\kappa_2)}{R^{1+\frac n2}}\ \Big (\int\limits_{ {\cal Q}_{\kappa_2
R}^K(0;t_0)}|u(x;t)|^2dxdt\Big )^{1/2},
\end{equation}
which holds for solutions to {\rm (\ref{0.5})} and
$1/2<\kappa_1<\kappa_2<1$ (see, e.g., \cite[Ch. III, Sect. 8]{LSU}).
\end{Rem}

Below we present some properties and estimates for $\lambda_c^{\pm}$ for
various geometries of ${\cal K}$.

\begin{sats}\label{lambda}
{\bf 1}. $\lambda_c^{\pm}>0$.\medskip

{\bf 2}. Let ${\cal K}_1\subset{\cal K}_2$. Then $\lambda_c^+({\cal K}_1,{\cal L})\ge\lambda_c^+({\cal K}_2,{\cal L})$,
$\lambda_c^-({\cal K}_1,{\cal L})\ge\lambda_c^-({\cal K}_2,{\cal L})$.\medskip

{\bf 3}. If $A_{ij}(t)\equiv\delta_{ij}$ then $\lambda_c^{\pm}({\cal K},{\cal L})=
\lambda_{\cal D}\equiv-\frac{m-2}{2}+\sqrt{\Lambda_{\cal D}+\frac{(m-2)^2}{4}}$, where
$\Lambda_{\cal D}$ is the first eigenvalue of the Dirichlet boundary value problem to
Beltrami-Laplacian in $\omega$.\medskip

{\bf 4}. Let $K$ be an acute cone i.e. $\overline K\setminus\{0\}\subset
\mathbb{R}^m_+ =\{x'\in\mathbb{R}^m:\ x^1>0\}$. Then
$\lambda_c^{\pm}>1$.\medskip

{\bf 5}. If $K\to\mathbb{R}^m_+$ then $\lambda_c^{\pm}\to1$. \medskip

{\bf 6}. Let ${\cal L}'=\partial_t-A_{k\ell}(t)D_kD_{\ell}$.
Then $\lambda_c^{\pm}({\cal K},{\cal L})=\lambda_c^{\pm}(K,{\cal L}')$.\medskip

{\bf 7}. $\lambda_c^{\pm}\geq -\frac{m}{2}+\nu\sqrt{\Lambda_{\cal D}+\frac{(m-2)^2}{4}}$.
In particular, if $\Lambda_{\cal D}\to\infty$ then $\lambda_c^{\pm}\to\infty$.\medskip

{\bf 8}. Suppose that $A_{ij}$ is a piece-wise constant matrix. Namely, let $-\infty=T_0<T_1<T_2<\dots<T_N=+\infty$, and let
$A_{ij}(t)\equiv A_{ij}^k$ for $t\in(T_{k-1},T_k)$, $k=1,\dots,N$. Then
$$\lambda_c^+({\cal K},{\cal L})=\lambda_c^-({\cal K},{\cal L})=\min\limits_{k=1,\dots,N}\widetilde\lambda_{\cal D}^{(k)}\equiv
-\frac{m-2}{2}+\sqrt{\min\limits_{k=1,\dots,N}\widetilde\Lambda_{\cal D}^{(k)}+\frac{(m-2)^2}{4}},
$$
where $\widetilde\Lambda_{\cal D}^{(k)}$ is the first eigenvalue of the Dirichlet boundary value problem to
Beltrami-Laplacian in the domain $\widetilde\omega^{(k)}=\widetilde K^{(k)}\cap {\mathbb S}^{m-1}$ while
the cones $\widetilde K^{(k)}$ are images of $K$ under change of variables reducing $A_{ij}^k$ to the canonical form.
\end{sats}


\begin{proof} It is sufficient to prove all the statements for $\lambda_c^+$.\medskip

{\bf 1}. Since $\partial \omega\in {\cal C}^{1,1}$, the
wedge ${\cal K}$ satisfies the so-called $A$-condition, see
\cite[Ch. I, Sect. 1]{LSU}. By \cite[Ch. III, Theorem 10.1]{LSU}, any solution of (\ref{0.5}) satisfies
the H\"older condition. Therefore, $\lambda_c^+>0$.

The only exception is the case of crack ($m=2$,
$\omega={\mathbb S}^1\setminus\{\Theta_0\}$, where $\Theta_0$ is a point on
${\mathbb S}^1$). In this case, one can transform $K$ to a half-plane by a
bi-Lipschitz mapping ($r\to r$ and $\theta\to \theta/2$,  where $r$
and $\theta$ are polar coordinates) and then apply \cite[Ch. III, Theorem 10.1]{LSU}.
Thus, in this case also $\lambda_c^+>0$.\medskip

{\bf 2}. This statement easily follows from the maximum principle.\medskip

{\bf 3}. For smooth $\omega$ this statement is well known (see, for
example \cite[Section 4]{Na}).
For  $\partial\omega\in{\cal C}^{1,1}$  we can approximate $\omega$
by smooth domains, and the required fact follows from the statement
{\bf 2}.\medskip

{\bf 4}. Without loss of generality we can assume $t_0=0$.
Since $K$ is acute, there exists $\theta<\frac {\pi}2$ such
that $K\subset K^{\theta}$. We
define the barrier function (\cite[Lemma 5.1]{Na1}, see also \cite[Lemma 4.1]{AN}) by
$$\widehat w(x;t)=\frac {(x^1)^{\gamma-1}}{\rho^{\gamma+1}}\cdot
\frac {(x^1)^2-\cos^2(\theta)|x'|^2}{\sin^2(\theta)}
+\frac{{\rm ctg}^2(\theta)\nu^2(x^1)^2}{4\rho^2}+\frac {|x''|^2-t}{(s_0\rho)^2}
$$
(here $s_0=\frac {(2(n-m)+\nu)^{\frac 12}}{{\rm ctg}(\theta)\nu}$).

Now we define $\widetilde\gamma=\widetilde\gamma(\nu,\theta)$
as a positive root of the quadratic equation
$\widetilde\gamma^2+\widetilde\gamma-{\rm ctg}^2(\theta)\nu^2=0$
and put $\gamma_*=\min\{1,\widetilde\gamma\}$.
Then direct calculation (see \cite[Lemma 5.1]{Na1}) shows that
if $0<\gamma\le\gamma_*(\nu,\theta)$ then ${\cal L}w\ge0$
on the set ${\cal Q}_{\rho,s_0\rho}^K$.
Moreover, $w\ge0$ on $\{x\in\partial{\cal K}\}$, and $w\ge
C(\theta,\nu)$ on the other parts of $\partial'{\cal Q}_{\rho,s_0\rho}^K$.

Now we choose $\rho=\rho(m,n,\nu,\theta)$
such that  ${\cal Q}_{\rho,s_0\rho}^K\subset {\cal Q}^K_{R/4}$ and consider
arbitrary solution of ${\cal L}u=0$ in ${\cal Q}^K_R$
vanishing for $\{x\in\partial{\cal K}\}$. By the maximum principle, for any
$\xi=(0,\xi'')$ with $|\xi''|<R/2$ and for any $\tau\in(0,R^2/4)$
the inequality
 $$\pm u(x+\xi;t-\tau)\le C^{-1}\widehat w(x;t)\cdot \sup\limits_{{\cal Q}^K_{3R/4}}\,(\pm u)$$
holds in ${\cal Q}_{\rho,s_0\rho}^K$. Setting $x''=0$ and $t=0$ we obtain (\ref{0.4}) with
$\kappa=\frac 34$ and $\lambda=1+\gamma_*$.
This implies $\lambda_c^+\ge 1+\gamma_*$, and the statement follows.\medskip

{\bf 5}. By the statement {\bf 2}, it is sufficient to prove that
if $K= K^{\theta}$ and $\theta\to\frac {\pi}2$ then  $\lambda_c^+\to1$.

For any  $0<\gamma<1$, we have
$$\pm{\cal L}(x^1)^{1\mp\gamma}>\varepsilon(\gamma)>0 \quad\mbox{in}\quad
(B_1^m\cap\mathbb{R}^m_+)\times{\mathbb R}^{n-m}\times{\mathbb R} .$$
By continuity, for any $\gamma>0$ there exists $\delta>0$ such that the functions
$$
w^{\pm}(x')=\Big(|x'|\cdot\cos\Big[(1\pm\delta)
\cdot\arccos\Big(\frac {x^1}{|x'|}\Big)\Big]\Big)^{1\mp\gamma}
$$
satisfy
$$\pm{\cal L}w^{\pm}>\frac 12\varepsilon(\gamma) \quad\mbox{in}\quad
(B_1^m\cap K^{\theta_{\pm}})\times{\mathbb R}^{n-m}\times{\mathbb R}$$
with $\theta_{\pm}=\frac {\pi}2(1\pm\delta)$. Therefore, the functions
$$
\widehat w^+(x;t)=w^+(x'/\rho)+\frac {|x|^2-t}{(s_1\rho)^2};
\qquad \widehat w^-(x;t)=w^-(x'/\rho)-\frac {|x''|^2-t}{(s_1\rho)^2}
$$
satisfy $\pm{\cal L}w^{\pm}\ge0$ in
${\cal Q}_{\rho,s_1\rho}^{K^{\theta_{\pm}}}$
if $s_1=s_1(m,n,\nu,\gamma)$ is large enough.

It is easy to see that $w^+\ge0$ on $x\in\partial{\cal K}^{\theta_+}$,
and $w\ge C(\theta,\nu)$ on the other parts of
$\partial'{\cal Q}_{\rho,s_1\rho}^{K^{\theta_+}}$. As in the proof of
statement {\bf 4}, this implies $\lambda_c^+({\cal K}^{\theta_+}, {\cal L})\ge 1-\gamma$.

Further, $w^-\le0$ on all parts of $\partial'{\cal Q}_{\rho,s_1\rho}^{K^{\theta_-}}$ besides
$\{(x;t):\ |x'|=\rho\}$. If
$\rho=\rho(m,n,\nu,\gamma)$ is such that  ${\cal Q}_{\rho,s_1\rho}^{K^{\theta_-}}\subset Q^K_{3R/4}$ then,
by the normal derivative lemma, any positive solution of ${\cal L}u=0$ in
$Q^K_R$ vanishing for $x\in\partial{\cal K}_1$ satisfies the
inequality $|Du|>0$ on the set $\partial'{\cal Q}_{\rho,s_1\rho}^{K^{\theta_-}}
\cap\{(x;t)\,:\,x'\in\partial K^{\theta_-},\ |x'|=\rho\}$. Thus, there exists a constant $C$ such that
$w^-\le Cu$ on $\partial'{\cal Q}_{\rho,s_1\rho}^{K^{\theta_-}}\cap\{(x;t):\ |x'|=\rho\}$. By the maximum principle,
 $u\ge C^{-1}w^-$ in ${\cal Q}_{\rho,s_1\rho}^{K^{\theta_-}}$ and therefore
$\lambda_c^+({\cal K}^{\theta_-}, {\cal L})\le 1+\gamma$.

Since $\gamma$ is arbitrarily small, the statement
follows.\medskip

{\bf 6}. This statement is subtle and uses some properties of the Green function.
We underline that the proof of these properties does not use statements {\bf 6} and {\bf 7}.

The inequality $\lambda_c^+({\cal K},{\cal L})\leq\lambda_c^+(K,{\cal L}')$ is quite clear,
since the function $u(x';t)$ satisfying ${\cal L}'u=0$ is a solution also ${\cal L}u=0$ in
corresponding set. Let us prove the opposite inequality. We put
$\lambda=\min\{\lambda_c^+({\cal K},{\cal L});\lambda_c^+(K,{\cal L}')-1\}-\varepsilon$ with a small
 $\varepsilon>0$ (by the statement {\bf 1} one can assume that $\lambda>-1$).

Let $u$ satisfy (\ref{0.5}).
Without loss of generality we assume $t_0=0$. We rewrite the equation as ${\cal L}'u=f_1+f_2$, where
$$
f_1=\sum_{i=m+1}^n\sum_{j=m+1}^nA_{ij}(t)D_iD_ju,\qquad
f_2=2\sum_{i=1}^m\sum_{j=m+1}^nA_{ij}(t)D_iD_ju
$$
(we fix the variable $x''\in B^{n-m}_{R/2}$ and consider it as a parameter).

Let $\zeta=\zeta(\tau)$ be a smooth function on $\mathbb R_+$, which is
equal to $1$ for $\tau<1/2$ and $0$ for $\tau>3/4$.
We put $\chi(x';t)=\zeta(|x'|/R)\zeta(\sqrt{|t|}/R)$. Then
$$
{\cal L}'(\chi u)=\big(\chi f_1+u\cdot {\cal L}'\chi\big)+\big(\chi f_2-A_{k\ell}D_k\chi\,D_{\ell}u)=:F_1+F_2.
$$
Using the Green function $\Gamma_K'$ for the Dirichlet problem for the operator ${\cal L}'$
in $K\times\mathbb R$, we obtain
\begin{equation}\label{Aug16b}
(\chi u)(x';t)=\int\limits_{-\infty}^t\int\limits_K
\Gamma_K'(x',y';t,s)\big(F_1(y';s)+F_2(y';s))\,dy'ds.
\end{equation}

Let $\mu=1-\lambda+\varepsilon$. By Lemma \ref{Green}, the kernel
$${\cal T}(x',y';t,s)=\frac{|x'|^{\mu-2}}{|y'|^{\mu}}\cdot\Gamma_K'(x',y';t,s)
$$
satisfies the inequality (\ref{kernel_k}) with $p=\infty$, $r=2$, $\lambda_1=\lambda_c^+(K,{\cal L}')-2-\varepsilon$,
$\lambda_2=\lambda_c^-(K,{\cal L}')-\varepsilon$. Thus, Proposition \ref{L_p} provides the estimate
\begin{equation}\label{Aug17c}
\sup_{Q_{R/2}^K}|x'|^{\mu-2}|u(x';t)|\leq C\cdot\sup_{Q_{3R/4}^K}|x'|^{\mu}(|F_1(x';t)|+|F_2(x';t)|),
\end{equation}
where $C$ does not depend on $u$. Further, Lemma \ref{Laug19} gives for $(x;t)\in {\cal Q}_{3R/4}^K$
\begin{equation*}\label{Aug17d}
|F_1(x';t)|\le \frac {C}{R^2}\Big (\frac{|x'|}{R}\Big )^{\lambda}\sup_{ {\cal Q}_{7R/8}^K}|u|;
\quad
|F_2(x';t)|\le \frac {C}{R^2}\Big (\frac{|x'|}{R}\Big )^{\lambda-1}\sup_{{\cal Q}_{7R/8}^K}|u|.
\end{equation*}
Therefore, estimate (\ref{Aug17c}) implies
$$|u(x';t)|\leq C\Big (\frac{|x'|}{R}\Big )^{2-\mu}\sup_{Q_{7R/8}^K}|u|\qquad\mbox{for}\quad
(x';t)\in {\cal Q}_{R/2}^K.
$$

Since all above estimates do not depend on $x''\in B^{n-m}_{R/2}$,
this ensures $\lambda_c^+({\cal K},{\cal L})\geq\lambda+1$. Since
$\varepsilon$ in the definition of $\lambda$ is arbitrarily small,
the statement follows.\medskip

{\bf 7}. By the statement {\bf 6} it's sufficient to prove this assertion for
$m=n$, i.e. ${\cal K}=K$. Let $u\in {\cal V}_\loc({\cal Q}_R^K(0;t_0))$
satisfy (\ref{Jan1a}) (without loss of generality, $t_0=0$).
Then the following local estimate is valid (cf. (\ref{local})):
\begin{equation}\label{Juni11b}
|u(x';t)|^2\leq C \rho^{-m-2}\int\limits_{{\cal P}_\rho}|u(y';\tau)|^2dy'd\tau,\qquad
 (x';t)\in{\cal Q}^K_{\frac 12},
\end{equation}
where $\rho=|x'|$ and
$${\cal P}_\rho={\cal Q}^K_{\frac {3\rho}2,\frac {\rho}2}(0;t)\setminus{\cal Q}^K_{\frac {\rho}2}(0;t)=
\Big\{ (y';\tau):\, y'\in K, \frac \rho 2<|y'|<\frac {3\rho}2, \tau\in (t-\frac {\rho^2}4,t)\Big\}.
$$

From (\ref{Juni11b}), it follows that for any $\mu\in\mathbb R$
$$
|u(x';t)|^2\leq C\rho^{-m-2\mu}\int\limits_{{\cal P}_\rho}|y'|^{2\mu-2}
|u(y';\tau)|^2dy'd\tau.
$$
Now, using the estimate (\ref{May9c}), we obtain
$$
|u(x';t)|^2\leq C\rho^{-m-2\mu}R^{2\mu-2}\int\limits_{{\cal Q}_{3R/4}^K} |u(y';\tau)|^2dy'd\tau\leq
C\Big (\frac{|x'|}{R}\Big)^{-m-2\mu}\sup_{{\cal Q}_{3R/4}^K}|u|^2
$$
for any $\mu$ satisfying the assumption of Lemma \ref{int_est}.
Thus, the statement {\bf 7} follows.\medskip

{\bf 8}. By the statement {\bf 6} we can assume that $m=n$. For
given $t_0\in\mathbb{R}$ and $R>0$, we introduce
$$h^2=\max \limits_{k=1,\dots,N}|(T_k,T_{k-1})\cap (t_0-R^2, t_0-R^2/4)|.
$$
Let, for example, $\min\{t_0-R^2/4,T_j\}-\max\{t_0-R^2,T_{j-1}\}=h^2$. We fix arbitrary
$0<\lambda<\min\limits_{k=1,\dots,N}\widetilde\lambda_{\cal D}^{(k)}$ and conclude by the statement {\bf 3}
that
\begin{equation}\label{qq}
|u(x';t_1)|\leq C(\lambda, \frac hR) \Big (\frac{|x'|}{h}\Big )^\lambda\sup_{{\cal Q}_{\kappa R}^K(0;t_0)}|u|
\quad\mbox{for}\quad x'\in {\cal B}_{R/2}^K,
\end{equation}
where $t_1=\min\{t_0-R^2/4,T_j\}$, $\kappa=\sqrt{1-h^2/R^2}$. Moreover, Proposition \ref{Pr1a} shows that
\begin{equation}\label{qqq}
|u(x';t)|\leq C(\lambda,\frac hR)\cdot r_x\sup_{{\cal Q}_{\kappa R}^K(0;t_0)}|u|
\quad\mbox{for}\quad x'\in {\cal B}_{R/2}^K,\ \ t\in (t_1,t_0).
\end{equation}

Note that the equation $\partial_tu-A_{ij}^kD_iD_ju=0$, $k=1,\dots, N$, has a solution
$w_k(x';t)=|x'|^{\lambda_{\cal D}^{(k)}}\Psi_k\Big(\frac{x'}{|x'|}\Big)$.
Therefore, function $\widetilde w_k(x';t)=|x'|^{\lambda}\Psi_k\Big(\frac{x'}{|x'|}\Big)$ is
a supersolution of this equation. The relations (\ref{qq}) and (\ref{qqq}) and the maximum principle
show that
\begin{equation}\label{qqqq}
|u(x';t)|\leq C(\lambda,\frac hR)\cdot \max\limits_{k=1,\dots,N}\{\widetilde w_k(x';t)\}\cdot h^{-\lambda}
\sup_{{\cal Q}_{\kappa R}^K(0;t_0)}|u|
\end{equation}
for $x'\in {\cal B}_{R/2}^K$, $t\in (t_1,t_0)$. Since $h>C(N)R$, the statement {\bf 8} follows.
\end{proof}

\begin{Rem} It will be interesting to study the question about
validity of property {\bf 8} from Theorem \ref{lambda} for infinite number
of layers, i.e. in the case
$$\dots<T_{-2}<T_{-1}<T_0<T_1<T_2<\ldots,\qquad T_{\pm k}\to{\pm\infty}\quad\mbox{as}\quad k\to+\infty, 
$$
and $A_{ij}(t)$ is constant on intervals $(T_k,T_{k+1})$, $k\in\mathbb Z$.
\end{Rem}

\section{Estimate of Green's function in ${\cal K}\times\mathbb R$}\label{Greenest}

\subsection{Local estimates of solutions}\label{SecR1}

The following statement is quite standard.

\begin{lem}\label{Lem9k}
 {\rm (i)} Let $u\in {\cal V}_\loc(Q_R(x_0;t_0))$ solve the equation ${\cal L}u=0$ in $Q_R(x_0;t_0)$.
Then for any multi-index $\alpha$ the function $D^{\alpha}u$ belongs to ${\cal V}_\loc(Q_R(x_0;t_0))$ and also
solves ${\cal L}u=0$ in $Q_R(x_0;t_0)$.\medskip

{\rm (ii)} Let $u\in {\cal V}_\loc ({\cal Q}_{R}^K(0;t_0))$  and satisfy {\rm (\ref {0.5})}.
Then for any multi-index $\alpha$ with $\alpha'=0$ the function
$D^{\alpha}u$ belongs to ${\cal V}_\loc({\cal Q}_{R}^K(0;t_0))$ and satisfies {\rm (\ref {0.5})}.
\end{lem}

The next statement can be found (even for more general equations) in  \cite[Ch. III, Sect. 11
and 12]{LSU}.

\begin{prop}\label{Pr1a}
Let $u\in {\cal V}(Q_R(x_0;t_0))$ solve the equation ${\cal L}u=0$ in $Q_R(x_0;t_0)$. Then
$$
|D u|\le \frac{C}{R}\sup_{Q_{R}(x_0;t_0)} |u|\qquad \mbox{in}\quad
Q_{R/2}(x_0;t_0).
$$
\end{prop}

Iterating this inequality we arrive at

\begin{kor}\label{Korol1a} Let  $u\in {\cal V}(Q_R(x_0;t_0))$ solve the equation ${\cal L}u=0$  in $Q_R(x_0;t_0)$. Then
$$
|D^\alpha u|\le \frac{C}{R^{|\alpha|}}\sup_{Q_{R}(x_0;t_0)}
|u|\qquad \mbox{in}\quad Q_{R/2^{|\alpha|}}(x_0;t_0).
$$
\end{kor}

The next statement can be extracted from \cite[Ch. III, Sect. 12]{LSU} and \cite[Sect. 3]{KN}.

\begin{prop}\label{L1aa}
For sufficiently small $\delta>0$, depending only on $K$,
the following assertion is valid. Let $x_0\in {\cal K}$, $r_{x_0}<\delta$ and $R\le |x_0'|/2$.
Suppose that $u\in {\cal V}({\cal Q}_R^K(x_0;t_0))$
 solves the equation ${\cal L}u=0$  in ${\cal Q}_R^K(x_0;t_0)$, and $u(x;t)=0$
for $x\in \partial{\cal K}$. Then
\begin{equation}\label{Kop1}
|D u (x;t)|\le
\frac{C}{R}\sup_{{\cal Q}_{R}^K(x_0;t_0)} |u|\qquad
\mbox{in}\quad  {\cal Q}_{R/8}^K(x_0;t_0),
\end{equation}
where   $C$ depends on $\nu$, $K$ and $\delta$.

Further, for $|\alpha'|\ge 2$ and arbitrary small $\varepsilon>0$
\begin{equation}\label{Kop1a}
d(x)^{|\alpha'|-2+\varepsilon}|D^{\alpha'} u(x;t)|\le
\frac{C}{R^{2-\varepsilon}}\sup_{{\cal Q}_{R}^K(x_0;t_0)} |u|\qquad
\mbox{in}\quad {\cal Q}_{R/8^{|\alpha'|}}^K(x_0;t_0),
\end{equation}
where   $C$ depends on $\nu$, $|\alpha'|$, $K$, $\delta$ and
$\varepsilon$.
\end{prop}

Taking into account Lemma \ref{Lem9k}, we obtain the following estimate.

\begin{kor}\label{Korol1b}
Let assumptions of {\rm Proposition \ref{L1aa}} be valid. If
$|\alpha'|\le1$, then
\begin{equation}\label{Kop1k}
|D^{\alpha} u (x;t)|\le
\frac{C}{R^{|\alpha|}}\sup_{{\cal Q}_{R}^K(x_0;t_0)} |u|\qquad
\mbox{in}\quad  {\cal Q}_{R/8^{|\alpha|}}^K(x_0;t_0),
\end{equation}
where $C$ depends on $\nu$, $|\alpha|$, $K$ and $\delta$.

If $|\alpha'|\ge 2$, then for arbitrary small $\varepsilon>0$
\begin{equation}\label{Kop1ak}
d(x)^{|\alpha'|-2+\varepsilon}|D^{\alpha} u(x;t)|\le
\frac{C}{R^{|\alpha''|+2-\varepsilon}}\sup_{{\cal Q}_{R}^K(x_0;t_0)}
|u|\qquad \mbox{in}\quad {\cal Q}_{R/8^{|\alpha|}}^K(x_0;t_0),
\end{equation}
where   $C$ depends on $\nu$,
$|\alpha|$, $K$, $\delta$ and $\varepsilon$.
\end{kor}

\begin{lem}\label{Laug19}
Let $u\in {\cal V}_\loc ({\cal Q}_{R}^K(0;t_0))$  and satisfy {\rm
(\ref {0.5})}. Let also $\lambda<\lambda_c({\cal K},{\cal L})$. Then
for any multi-index $\alpha$ such that $|\alpha'|\le1$ we have
\begin{equation}\label{0.4x}
|D^{\alpha}u(x;t)|\leq \frac {C(\lambda)}{R^{|\alpha|}}
\Big (\frac{|x'|}{R}\Big )^{\lambda-|\alpha'|}\sup_{ {\cal Q}_{3R/4}^K(0;t_0)}|u|
\quad\mbox{for}\quad (x;t)\in {\cal Q}_{R/2}^K(0;t_0).
\end{equation}
\end{lem}

\begin{proof}
First, let $\alpha'=0$. By Lemma \ref{Lem9k}, $D^{\alpha}u\in {\cal V}_\loc({\cal Q}_{R}^K(0;t_0))$.
Using the definition of the critical exponent and Remark \ref{Rem1}, we obtain
for $(x;t)\in {\cal Q}_{R/2}^K(0;t_0)$
\begin{equation}\label{0.4b}
|D^{\alpha}u(x;t)|\leq \frac {C(\lambda,\kappa)}{R^{1+\frac n2}}\ \Big(\frac{|x'|}{R}\Big )^\lambda
\Big (\int\limits_{ {\cal Q}_{\kappa R}^K(0;t_0)} |D^{\alpha}u(y;\tau)|^2dy\,d\tau\Big )^{1/2}
\end{equation}
with $\kappa\in (1/2,1)$.
The right-hand side of (\ref{0.4b}) can be estimated via iterations of the local estimate
(see, e.g., \cite[Ch. III, Sect. 2]{LSU})
\begin{equation*}
\|u\|_{{\cal V}({\cal Q}_{\kappa_1 R}^K(0;t_0))}\leq \frac {C(\kappa_1,\kappa_2)}R\,
\|u\|_{L_2({\cal Q}_{\kappa_2 R}^K(0;t_0))}
\end{equation*}
(here $\kappa_1<\kappa_2$). This gives
$$
|D^{\alpha}u(x;t)|\leq \frac {C(\lambda)}{R^{1+\frac n2+|\alpha|}}\ \Big
(\frac{|x'|}{R}\Big )^\lambda\Big (\int\limits_{ {\cal Q}_{7R/8}^K(0;t_0)}|u(y;\tau)|^2dy\,d\tau\Big )^{1/2},
$$
which implies (\ref{0.4x}).\medskip

For $|\alpha'|=1$ we apply Proposition \ref{Pr1a} and the estimate (\ref{Kop1}) to $D^{\alpha''}u$ and
use obtained estimate (\ref{0.4x}) for $\alpha'=0$.
\end{proof}

\subsection{Green's function in ${\cal K}\times\mathbb R$}

 Let us consider equation (\ref{Jan1}) in the whole space.
Using the Fourier transform with respect to $x$ we obtain the
following representation of solution through the right-hand side:
\begin{equation}\label{TTN1}
u(x;t)=\int\limits_{-\infty}^t\int\limits_{{\mathbb R}^n}
\Gamma(x,y;t,s) f(y;s)\ ds,
\end{equation}
where $\Gamma$ is the Green function of the operator ${\cal L}$
 given by
\begin{equation*}
\Gamma(x,y;t,s)= \frac { \det\big(\int\limits_s^t
A(\tau)d\tau\big)^{-\frac 12}} {(4\pi)^{\frac n2}}    \exp
\bigg(-\frac {\Big(\big(\int\limits_s^t A(\tau) d\tau\big)^{-1}
(x-y),(x-y)\Big)}4\bigg)
\end{equation*}
for $t>s$ and $0$ otherwise.  Here by $A(\tau)$ is denoted the matrix
$\{ A_{ij}(\tau)\}$. The above representation implies, in particular,
 the following estimates for $\Gamma$
\begin{equation}\label{TTN1h}
\Big |\partial_t^k D_x^\alpha D_y^\beta \Gamma(x,y;t,s)\Big |\leq
\frac{C({k,\alpha,\beta})}{(t-s)^{(n+2k+|\alpha|+|\beta|)/2}}\,\exp\left(-\frac{\sigma|x-y|^2}{t-s}\right),
\end{equation}
where $k\leq 1$ and $\alpha$ and $\beta$ are arbitrary
multi-indices. Here $\sigma$ is a positive constant depending on
$\nu$.\medskip

Departing from $\Gamma$ we can construct the Green function for the
Dirichlet problem in ${\cal K}\times\mathbb R$.

\begin{lem}\label{Greenexist}
For all $(y;s)\in{\cal K}\times \mathbb R$ there is a function
$u(x;t)=\Gamma_{\cal K}(x,y;t,s)$ satisfying
the Dirichlet problem
\begin{equation}\label{Ra1a}
{\cal L}u=\delta(x-y)\delta(t-s)\quad\mbox{in}\quad {\cal K}\times \mathbb R;\qquad
u\big|_{x\in\partial {\cal K}}=0;\qquad u\big|_{t<s}=0
\end{equation}
(the equation is understood in the sense of distributions). This
function admits the representation
\begin{equation}\label{TTN1am}
\Gamma_{\cal K}(x,y;t,s)=\chi(x;t)\Gamma
(x,y;t,s)+\widetilde\Gamma(x,y;t,s),
\end{equation}
where $\chi\in C_0^\infty ({\cal K}\times\mathbb R)$  is equal to
$1$ in a neighborhood of $(y;t)$ and the function
$w(x;t)=\widetilde\Gamma(x,y;t,s)$ belongs to ${\cal V}({\cal
K}\times \mathbb R)$. Moreover, the function $\Gamma_{\cal K}$
satisfies the estimate for $x,y\in{\cal K}$ and $t>s$:
\begin{equation}\label{TTN1a}
0\le\Gamma_{\cal K}(x,y;t,s)\leq C(t-s)^{-\frac n2}\,\exp\left(-\frac{\sigma|x-y|^2}{t-s}\right).
\end{equation}
\end{lem}

\begin{proof}
Let $\zeta=\zeta(\tau)$ be a smooth function on $\mathbb R_+$, which is
equal to $1$ for $\tau<1/2$ and $0$ for $\tau>3/4$.
We put $\chi(x;t)=\zeta(|x-y|/d(y))\zeta(\sqrt{|t-s|}/d(y))$. Then
$\chi\in C_0^\infty ({\cal K}\times\mathbb R)$. For sufficiently large $R$ we introduce the notation
$\widetilde{\mathfrak Q}_R={\cal B}^K_R\times(s,\infty)$ and consider the Dirichlet problem
\begin{equation*}
{\cal L}u=-\partial_t\chi\,\Gamma+2A_{ij}D_i\chi\,D_{x_j}\Gamma+\Gamma A_{ij}D_iD_j\chi
\quad\mbox{in}\quad \widetilde{\mathfrak Q}_R;\qquad u\big|_{x\in\partial' \widetilde{\mathfrak Q}_R}=0.
\end{equation*}
This problem has a unique weak solution in the space ${\cal V}(\widetilde{\mathfrak Q}_R)$, see, e.g.,
\cite[Ch. III, Sect. 3 and 4]{LSU}. We denote this solution $\widetilde\Gamma_R(x,y;t,s)$, expand it by zero
for $t<s$ and consider the function
$$
\Gamma_{{\cal K},R}(x,y;t,s)=\chi(x;t)\Gamma (x,y;t,s)+\widetilde\Gamma_R(x,y;t,s).
$$
This function obviously solves the problem
\begin{equation*}
{\cal L}u=\delta(x-y)\delta(t-s)\quad\mbox{in}\quad {\cal B}^K_R\times \mathbb R;\qquad
u\big|_{x\in\partial {\cal B}^K_R}=0;\qquad u\big|_{t<s}=0.
\end{equation*}
It is easy to check that the definition of $\Gamma_{{\cal K},R}$
does not depend on the choice of the cut-off function. Using maximum
principle, we obtain that $\Gamma_{{\cal K},R}$ increases with
respect to $R$ for all values of arguments and satisfies
$0\le\Gamma_{{\cal K},R}(x,y;t,s)\leq \Gamma (x,y;t,s)$. Therefore,
there exists the limit function $\Gamma_{\cal K}$ satisfying
(\ref{Ra1a}) and (\ref{TTN1am}) with $\widetilde\Gamma\in {\cal
V}({\cal K}\times \mathbb R)$. The estimate (\ref{TTN1a}) follows from
(\ref{TTN1h}).
\end{proof}

Similarly one can construct the Green function
$u(x;t)=\widehat{\Gamma}_{\cal K}(x,y;t,s)$ for the  operator $\widehat{\cal L}$
defined in (\ref{0.51}). This function satisfies
\begin{equation*}
\widehat{\cal L}u=\delta(x-y)\delta(t-s)\quad\mbox{in}\quad {\cal K}\times \mathbb R;
\qquad u\big|_{x\in\partial {\cal K}}=0;\qquad u\big|_{t<s}=0.
\end{equation*}
By changing of variables $t\to -t$ and $s\to -s$, we conclude that the function
$v(x;t)=\widehat{\Gamma}_{\cal K}(x,y;-t,-s)$ solves the problem
\begin{equation*}
{\cal L}^*v=\delta(x-y)\delta(t-s)\quad\mbox{in}\quad {\cal K}\times
\mathbb R;\qquad v\big|_{x\in\partial {\cal K}}=0;\qquad
v\big|_{t>s}=0,
\end{equation*}
where $ {\cal L}^*=-\partial_t-A_{ij}(t)D_iD_j$ is operator formally adjoint to ${\cal L}$.

Using Green's formula for the functions $u(x;t)=\Gamma_{\cal K}(x,y_1;t,s_1)$ and
$v(x;t)=\widehat{\Gamma}_{\cal K}(x,y_2;-t,-s_2)$ with $s_1<t<s_2$, we get
$$
v(y_1;s_1)=\int\limits_{\mathbb R}\int\limits_{\cal K}{\cal L}u\cdot v\,dxdt=
\int\limits_{\mathbb R}\int\limits_{\cal K}u\cdot {\cal L}^*v\,dxdt=u(y_2;s_2),
$$
which means
\begin{equation}\label{2Okt1a}
{\Gamma}_{\cal K}(y,x;-s,-t)=\widehat{{\Gamma}}_{\cal K}(x,y;t,s).
\end{equation}

\begin{Rem}\label{Rem3}
In what follows, we obtain various pointwise estimates for the Green function.
These estimates depend only on the ellipticity constant $\nu$ and the critical exponents
$\lambda_c^{\pm}$; moreover, they contain $t$ and $s$ only in the combination $t-s$.
Therefore, such estimates are valid for Green's functions of operators $\cal L$ and
$\widehat {\cal L}$ simultaneously (with interchange of $\lambda_c^+$ and $\lambda_c^-$).

Note that in our paper \cite{KN} in the proofs of Lemma 4, Theorem 3 and Lemma 5 we
argued that the Green function (in the half-space) is symmetric with respect to $x$ and $y$.
In fact, this is not the case. However, all corresponding estimates are true just by
the statement of the previous paragraph, since in the case 
$m=1$ considered in \cite{KN} all the estimates depend only on $\nu$.
\end{Rem}

Note also that, since the operators ${\cal L}$ and $\widehat{\cal L}$ are invariant with
respect to translations along the edge, $\Gamma_{\cal K}$ and $\widehat{\Gamma}_{\cal K}$
depend actually on $x'$, $y'$ and the difference $x''-y''$.\medskip

We shall use the notations
$$
{\cal R}_x=\frac{|x'|}{|x'|+\sqrt{t-s}};\qquad
{\cal R}_y=\frac{|y'|}{|y'|+\sqrt{t-s}}.
$$

\begin{lem}\label{Green} For $\lambda^+<\lambda_c^+$ and $\lambda^-<\lambda_c^-$ the following inequality
\begin{equation}\label{May0}
| \Gamma_{\cal K}(x,y;t,s)|\le C\,{\cal R}_x^{\lambda^+}{\cal R}_y^{\lambda^-}(t-s)^{-\frac n2}
\,\exp\left(-\frac{\sigma_1|x-y|^2}{t-s}\right)
\end{equation}
holds for $x,y\in{\cal K}$ and $t>s$. Here $\sigma_1$ is a positive
constant depending only on the ellipticity constant $\nu$ and $C$
may depend on $\nu$ and $\lambda^\pm$.
\end{lem}

\begin{proof} It is sufficient to prove the estimates for $s=0$ and
$x''=0$. First, let us prove the inequality
\begin{equation}\label{Oct24a}
\Gamma_{\cal K}(x,y;t,0)\leq C{\cal R}_x^{\lambda^+}\ t^{-\frac n2}
\exp\left(-\frac{\sigma_1|x-y|^2}{t}\right).
\end{equation}
If $|x'|^2\geq t/3$ then this estimate easily follows from (\ref{TTN1a}).
Consider the case $|x'|^2<t/3$. We take $R=\sqrt{t/2}$ and apply (\ref{0.4})
with $t_0=t$ to the function $u(x;t)=\Gamma_{\cal K}(x,y;t,0)$. This gives
\begin{equation}\label{Oct24b}
\Gamma_{\cal K}(x,y;t,0)\leq C\Big (\frac{|x'|}{\sqrt{t}}\Big)^{\lambda^+}
\sup_{{\cal Q}_{\kappa R}^K(0;t)}\Gamma_{\cal K}(z,y;\tau,0).
\end{equation}
From (\ref{TTN1a}) it follows that for $(z,\tau)\in {\cal Q}_{\kappa R}^K(0;t)$
the inequality
$$
\Gamma_{\cal K}(z,y;\tau,0)\leq C t^{-\frac n2}\,\exp\left(-\frac{\sigma_1|x-y|^2}{t}\right)
$$
holds with a certain positive $\sigma_1\leq\sigma$. Using this for
estimating the right-hand side in (\ref{Oct24b}), we obtain
(\ref{Oct24a}) and hence (\ref{May0}) for $\lambda^-=0$. By Remark
\ref{Rem3}, we get a similar to (\ref{Oct24a}) estimate for
$\widehat{\Gamma}_{\cal K}$:
\begin{equation}\label{Oct24ag}
\widehat{\Gamma}_{\cal K}(x,y;t,\tau)\leq C{\cal R}_x^{\lambda^-}\
(t-\tau)^{-\frac n2}
\exp\left(-\frac{\sigma_1|x-y|^2}{t-\tau}\right).
\end{equation}
Due to (\ref{2Okt1a}) this inequality leads to
\begin{equation}\label{Okt25a}
\Gamma_{\cal K}(x,y;t,\tau)\leq C{\cal R}_y^{\lambda^-}\
(t-\tau)^{-\frac n2}
\exp\left(-\frac{\sigma_1|x-y|^2}{t-\tau}\right).
\end{equation}
 Now we use the same arguments as above in order to prove
(\ref{May0}) in general case but instead of inequality
(\ref{TTN1a}) we use  (\ref{Oct24ag}) for estimating of supremum
in the right-hand side of (\ref{Oct24b}). The proof is complete.
\end{proof}

\begin{sats}\label{T2s} Let $\lambda^+<\lambda_c^+$, $\lambda^-<\lambda_c^-$,
 and let $|\alpha'|, \, |\beta'|\leq 2$.
For $x,y\in {\cal K}$, $t>s$ the following estimates are valid
\begin{eqnarray}\label{May3}
&&|D_{x}^{\alpha}D_{y}^{\beta} \Gamma_{\cal K}(x,y;t,s)|\le
C\,{\cal R}_x^{\lambda^+-|\alpha'|} {\cal
R}_y^{\lambda^--|\beta'|}r_x^{-\varepsilon}r_y^{-\varepsilon}
\nonumber\\
&&{(t-s)^{-\frac{n+|\alpha|+|\beta|}2}} \,\exp
\left(-\frac{\sigma_1|x-y|^2}{t-s}\right),
\end{eqnarray}
\begin{eqnarray}\label{May4}
&&|D_{x}^{\alpha}\partial_s \Gamma_{\cal K}(x,y;t,s)|\le C\,{\cal
R}_x^{\lambda^+-|\alpha'|} {\cal
R}_y^{\lambda^--2}r_x^{-\varepsilon}r_y^{-\varepsilon}
\nonumber\\
&&{(t-s)^{-\frac {n+|\alpha|+2}2}} \,\exp
\left(-\frac{\sigma_1|x-y|^2}{t-s}\right),
\end{eqnarray}
where $\sigma_1$ is a positive constant depending on $\nu$,
$\varepsilon$ is an arbitrary small positive number  and $C$ may
depend on $\nu$, $\lambda$, $\alpha$, $\beta$ and $\varepsilon$. If
$|\alpha'|\le 1$ {\rm(}or $|\beta'|\le 1${\rm)} then the factor
$r_x^{-\varepsilon}$ (respectively, $r_y^{-\varepsilon}$) must be
removed from the right-hand side.
\end{sats}

\begin{proof} It is sufficient to prove Theorem for $s=0$.
We start with the case $\beta=0$ in (\ref{May3}) and  consider two alternatives:
$|\alpha'|\le1$ and $|\alpha'|=2$.

Let $|\alpha'|\le1$. First suppose that $t\geq 4|x'|^2$.
Using Lemma \ref{Laug19} with $R=\sqrt{t}$, we obtain
$$
|D_x^{\alpha}\Gamma_{\cal K}(x,y;t,0)|\leq Ct^{-\frac
{|\alpha''|+\lambda^+}2}|x'|^{\lambda^+-|\alpha'|}
\sup_{(z,\tau)\in{\cal Q}_{3R/4}^K(x;t)}\Gamma_{\cal
K}(z,y;\tau,0).
$$
 By (\ref{May0})
$$
\Gamma_{\cal K}(z,y;\tau,0)\leq C {\cal R}_y^{\lambda^-}\
t^{-\frac n2} \exp\Big(-\sigma_2\frac{|x-y|^2}{t}\Big )\quad
\mbox{for}\quad (z,\tau)\in{\cal Q}_{3R/4}^K(x;t)
$$
with some positive $\sigma_2$. Therefore
$$
|D_x^{\alpha}\Gamma_{\cal K}(x,y;t,0)|\leq C {\cal
R}_x^{\lambda^+-|\alpha'|} {\cal R}_y^{\lambda^-}\ t^{-\frac
{n+|\alpha|}2} \exp\Big(-\sigma_2\frac{|x-y|^2}{t}\Big ),
$$
which gives (\ref{May3}) in the case $\beta=0$, $|\alpha'|\le1$ and $t\geq 4|x'|^2$.\medskip

Let now $t\leq 4|x'|^2$. Consider two subcases.\medskip

{\it Subcase 1}: $r_x\ge\delta$.
We put $R=\min \{\delta|x'|/2,\sqrt{t}/2\}$. Then the set $Q_R(x;t)$ belongs to
${\cal K}\times\mathbb R$. Applying the local estimate from Corollary \ref{Korol1a}
to the function $u(x;t)=\Gamma_{\cal K}(x,y;t,0)$, we obtain
\begin{equation}\label{Aug27a}
|D^\alpha_x \Gamma_{\cal K}(x,y;t,0)|\le
\frac{C}{R^{|\alpha|}}\sup_{(z,\tau)\in Q_{R}(x;t)} \Gamma_{\cal
K}(z,y;\tau,0),
\end{equation}
here $\alpha$ can be arbitrary multi-index. Using inequality
(\ref{May0}) to estimate the Green function in the right-hand
side, we obtain
\begin{equation}\label{Aug20a}
|D^\alpha_x \Gamma_{\cal K}(x,y;t,0)|\leq C {\cal
R}_x^{\lambda^+}{\cal R}_y^{\lambda^-} R^{-|\alpha|}t^{-\frac
n2}\exp\left(-\frac{\sigma_2|x-y|^2}{t}\right),
\end{equation}
which gives (\ref{May3}) in this subcase.\medskip

{\it Subcase 2}: $r_x<\delta$.
Using Corollary \ref{Korol1b} for the Green function, we get
\begin{equation}\label{Oct2sg}
|D^{\alpha}_x \Gamma_{\cal K}(x,y;t,0)|\le
\frac{C}{R^{|\alpha|}}\sup_{(z,\tau)\in {\cal Q}_{R}^K(x;t)}|\Gamma_{\cal K}(z,y;\tau,0)|,
\end{equation}
where $R=\sqrt{t}/4$. Using inequality (\ref{May0}), we obtain
\begin{equation}\label{Oct2sag}
|D^{\alpha}_x \Gamma_{\cal K}(x,y;t,0)|\le C {\cal
R}_y^{\lambda^-}\ t^{-\frac
{n+|\alpha|}2}\exp\left(-\frac{\sigma_2|x-y|^2}{t}\right),
\end{equation}
which leads to (\ref{May3}) in this subcase.

Thus (\ref{May3}) is proved in the case $\beta=0$ and $|\alpha'|\le1$.\medskip

We turn to the second alternative $|\alpha'|=2$ and also consider two subcases.\medskip

{\it Subcase 1}: $r_x\ge\delta$. We again put $R=\min \{\delta|x'|/2,\sqrt{t}/2\}$ and
apply the local estimate from Lemma \ref{Korol1a} to the function
$u(x;t)=D_x^{\alpha''}\Gamma_{\cal K}(x,y;t,0)$ in $Q_R(x;t)$. This leads to
$$|D^\alpha_x \Gamma_{\cal K}(x,y;t,0)|\le
\frac{C}{R^2}\sup_{(z,\tau)\in Q_R(x;t)} |D_x^{\alpha''}\Gamma_{\cal K}(z,y;\tau,0)|,
$$
and the estimate (\ref{May3}) with $\alpha'=0$ and $\beta=0$ gives
\begin{equation}\label{Sent1}
|D_x^{\alpha} \Gamma_{\cal K}(x,y;t,0)|\le \frac{C}{R^2}\, {\cal
R}_x^{\lambda^+} {\cal R}_y^{\lambda^-}\ t^{-\frac{n+|\alpha''|}2}
\, \exp\left(-\frac{\sigma_1|x-y|^2}{t-s}\right).
\end{equation}

{\it Subcase 2}:  $r_x<\delta$. Using
(\ref{Kop1a}) for $u(x;t)=D_x^{\alpha''}\Gamma_{\cal K}(x,y;t,0)$ with $R=\min \{|x'|/4, \sqrt{t}/2\}$,
we obtain
\begin{equation*}\label{Oct2s}
d(x)^{\varepsilon}|D^\alpha_x \Gamma_{\cal K}(x,y;t,0)|\le
\frac{C}{R^{2-\varepsilon}}\sup_{(z,\tau)\in {\cal Q}_{R}^K(x;t)}|
D_x^{\alpha''}\Gamma_{\cal K}(z,y;\tau,0)|,
\end{equation*}
and the estimate (\ref{May3}) with $\alpha'=0$ and $\beta=0$ leads to
\begin{equation}\label{Oct2sa}
d(x)^{\varepsilon}|D^\alpha_x \Gamma_{\cal K}(x,y;t,0)|\le
\frac{C}{R^{2-\varepsilon}}{\cal R}_x^{\lambda^+} {\cal
R}_y^{\lambda^-}\ t^{-\frac
{n+|\alpha''|}2}\exp\left(-\frac{\sigma_1|x-y|^2}{t}\right).
\end{equation}
As (\ref{Sent1}) as (\ref{Oct2sa}) imply
\begin{equation*}
|D^\alpha_x \Gamma_{\cal K}(x,y;t,0)|\le \frac
{C|x'|^2t}{R^2(|x'|+\sqrt{t})^2}\ \frac {{\cal R}_x^{\lambda^+-2}
{\cal R}_y^{\lambda^-}}{r_x^{\varepsilon}\ t^{\frac
{n+|\alpha|}2}} \exp\left(-\frac{\sigma_1|x-y|^2}{t}\right).
\end{equation*}
The first quotient here is bounded, and we arrive at (\ref{May3}).

Thus inequality (\ref{May3}) is proved in the case $\beta=0$ and
$|\alpha'|\leq 2$. By Remark \ref{Rem3}, we obtain that
\begin{eqnarray}\label{May33a}
&&|D_{x}^{\alpha} \widehat{\Gamma}_{\cal K}(x,y;t,s)|\le C\,{\cal
R}_x^{\lambda^--|\alpha'|} {\cal
R}_y^{\lambda^+}r_x^{-\varepsilon}
\nonumber\\
&&{(t-s)^{-\frac{n+|\alpha|}2}} \,\exp
\left(-\frac{\sigma_1|x-y|^2}{t-s}\right).
\end{eqnarray}
From (\ref{May33a}) and (\ref{2Okt1a}) it follows that
\begin{eqnarray}\label{May33ab}
&&|D_{y}^{\beta} \Gamma_{\cal K}(x,y;t,s)|\le C\,{\cal
R}_x^{\lambda^+} {\cal R}_y^{\lambda^--|\beta'|}r_y^{-\varepsilon}
\nonumber\\
&&{(t-s)^{-\frac{n+|\beta|}2}} \,\exp
\left(-\frac{\sigma_1|x-y|^2}{t-s}\right).
\end{eqnarray}
This gives inequality (\ref{May3}) for $\alpha=0$ and
$|\beta'|\leq 2$.

 In general case $|\alpha'|\leq 2$ and $|\beta'|\leq 2$,
  we  repeat the above proof (when $\beta=0$)
but instead of inequality (\ref{May0}) we use (\ref{May33ab}).
This gives (\ref{May3}) with arbitrary $|\alpha'|\leq 2$ and
$|\beta'|\leq 2$.

Finally, inequality (\ref{May4}) follows from (\ref{May3}), since
the derivative with respect to $s$ can be expressed through the
second derivatives with respect to $y$.
\end{proof}

\begin{Rem}\label{init2}
To study the initial-boundary value problem {\rm (\ref{init})}, we
need the Green function $\Gamma_{\cal K}(x,y;t,s)$ only for
$t>s\ge0$. In this case the estimates in Lemma \ref{Green} and
Theorem \ref{T2s} hold true for $t>s\ge0$ if we define $\lambda^+_c$
and $\lambda^-_c$ as described in Remark \ref{init1} (taking into
account the relation {\rm (\ref{2Okt1a})}).
\end{Rem}

In what follows we denote by the same letter the kernel  and the
corresponding integral operator, i.e.
\begin{equation}\label{Feb15a}
({\cal G}h)(x;t)=\int\limits_{-\infty}^t\int\limits_{\mathbb R^n}
{\cal G}(x,y;t,s)h(y;s)\,dyds.
 \end{equation}
If necessary, all functions are assumed to be expanded by zero to the whole space-time.

\section{The weighted estimates in a wedge}\label{Rn+2}

\begin{sats}\label{tildeLpq}
Let $1<p,q<\infty$, and let $\mu$ be subject to {\rm (\ref{mu})}. Suppose also that
$|x'|^\mu f\in \widetilde{L}_{p,q}({\cal K}\times\mathbb R)$. Then the function
\begin{equation}\label{Sept12a}
u(x;t)=\int\limits_{-\infty}^t\int\limits_{\mathbb R^n} {\Gamma}_{\cal K}(x,y;t,s)f(y;s)\,dyds
\end{equation}
solves the problem {\rm (\ref{Jan1a})} and satisfies the estimate {\rm (\ref{tilde})}.
\end{sats}

\begin{proof}
It is easy to see that for $f\in{\cal C}^\infty_0({\cal K}\times\mathbb R)$ the function (\ref{Sept12a}) is a
solution of (\ref{Jan1a}). Therefore, it suffices to prove the estimate (\ref{tilde}).

First, we choose $0<\lambda^+<\lambda_c^+$ and $0<\lambda^-<\lambda_c^-$ such that
\begin{equation*}
2-\frac mp-\lambda^+<\mu<m-\frac mp+\lambda^-
\end{equation*}
and observe, that by (\ref{May0}) the kernel ${\cal G}=\frac{|x'|^{\mu-2}}{|y'|^\mu}\cdot\Gamma_{\cal K}(x,y;t,s)$
satisfies the assumptions of Proposition \ref{L_p} (with $\lambda_1=\lambda^+-2$, $\lambda_2=\lambda^-$, $r=2$).
Therefore, the operator ${\cal G}$ is bounded in $L_p(\mathbb R^n\times\mathbb R)$ and in
$\widetilde L_{p,\infty}(\mathbb R^n\times \mathbb R)$. Generalized Riesz--Thorin theorem, see, e.g.,
\cite[1.18.7]{Tr}, shows that this operator is bounded in $\widetilde L_{p,q}(\mathbb R^n\times\mathbb R)$
for any $q\ge p$. For $q<p$ this statement follows by duality arguments. Thus, we obtain the inequality
\begin{equation}\label{tilde-zero}
|\!|\!| |x'|^{\mu-2} u|\!|\!|_{p,q}\le C\ |\!|\!||x'|^\mu f|\!|\!|_{p,q}.
\end{equation}

Now we consider a point $\xi'' \in{\bf R}^{n-m}$, and, given $\rho>0$, $\gamma>1$ define
$$\widetilde{\cal P}_{\rho,\gamma}(\xi'')={\cal B}^K_{\rho}(0,\xi'')\setminus
{\cal B}^K_{\frac {\rho}\gamma,\rho}(0,\xi'')=\Big\{x\in{\cal K}\, :\, \frac{\rho}{\gamma}<
|x'| < \rho,\ |x''-\xi''|<\rho\Big\}.
$$
For any $\xi'' \in{\mathbb R}^{n-m}$ and $\rho>0$ the inequality
\begin{multline*}
\int\limits_{\widetilde{\cal P}_{\rho,2}(\xi'')}
\bigg( \int\limits_{\mathbb R}\left(|\partial_t u|^q+|D(Du)|^q\right)\, dt\bigg)^{p/q} dx \\
\le C\int\limits_{\widetilde{\cal P}_{2\rho,8}(\xi'')}\bigg( \int\limits_
{\mathbb R}\left(\rho^{-2q}|u|^q+|f|^q\right) \, dt\bigg)^{p/q} dx
\end{multline*}
easily follows from localization of the estimate \cite[Theorem 4]{KN}.

Using a proper partition of unity in ${\cal K}$, we arrive at
\begin{multline*}
\int\limits_{\cal K}\bigg(\int\limits_{\mathbb R}\left(|\partial_tu|^q+|D(Du)|^q
\right)\, dt\bigg)^{p/q}|x'|^{\mu p} dx\\
\le C\bigg(\int\limits_{\cal K}\bigg(\int\limits_{\mathbb R}|u|^q
\, dt\bigg)^{p/q}|x'|^{\mu p-2p} dx + \int\limits_{\cal K}\bigg(
\int\limits_{\mathbb R}|f|^q\, dt\bigg)^{p/q}|x'|^{\mu p} dx\bigg).
\end{multline*}
This immediately implies (\ref{tilde}) with regard of
(\ref{tilde-zero}).
\end{proof}

To deal with the scale $L_{p,q}$, we need the following lemma.

\begin{lem}\label{weak_2}
Let a function $h$ be supported in the layer $|s-s^0|\le\delta$ and
satisfy $\int h(y;s)\ ds\equiv 0$. Also let $p\in (1,\infty)$ and
$\mu$ be subject to {\rm (\ref{mu})}. Then the integral operators
${\cal G}$ and ${\cal G}_{ij}$ with kernels
$${\cal G}=\frac{|x'|^{\mu-2}}{|y'|^\mu}\cdot\Gamma_{\cal K}(x,y;t,s);\qquad
{\cal G}_{ij}=\frac{|x'|^\mu}{|y'|^\mu}\cdot
D_{x_i}D_{x_j}\Gamma_{\cal K}(x,y;t,s)$$ satisfy
$$\int\limits_{|t- s^0|>2\delta}\Vert ({\cal G}h)(\cdot; t)\Vert_p\ dt\le
C\,\Vert h\Vert_{p,1};\qquad \int\limits_{|t- s^0|>2\delta}\Vert
({\cal G}_{ij}h)(\cdot; t)\Vert_p\ dt\le C\,\Vert h\Vert_{p,1},$$

\noindent where $C$ does not depend on $\delta$ and $ s^0$.
\end{lem}

\begin{proof}
By $\int h(y; s)\, ds\equiv 0$, we have
\begin{equation}\label{difference1}
({\cal G}h)(x;t)=\int\limits_{-\infty}^{t}\int\limits_{\cal K} \Bigl({\cal
G}(x,y;t, s)-{\cal G}(x,y;t,s^0)\Bigr)\, h(y; s)\, dy\, ds;
\end{equation}
\begin{equation}\label{difference2}
({\cal G}_{ij}h)(x;t)=\int\limits_{-\infty}^{t}\int\limits_{\cal K}
\Bigl({\cal G}_{ij}(x,y;t, s)-{\cal G}_{ij}(x,y;t,s^0)\Bigr)\, h(y;
s)\, dy\, ds.
\end{equation}
We choose $0<\lambda^+<\lambda_c^+$, $0<\lambda^-<\lambda_c^-$ and
$0<\varepsilon<\min\{\frac 1p,1-\frac 1p\}$ such that
\begin{equation}\label{mueps}
2-\frac mp-\lambda^+<\mu<m-\frac mp+\lambda^--2\varepsilon.
\end{equation}
For $| s- s^0|<\delta$ and $t- s^0>2\delta$, estimates (\ref{May4})
with $|\alpha|=0$ and $|\alpha|=2$ imply
\begin{eqnarray*}
&&\left|{\cal G}(x,y;t, s)-{\cal G}(x,y;t, s^0)\right|
\le\int\limits_{s^0}^s|\partial_\tau {\cal G}(x,y;t,\tau)|\,d\tau\\
&&\le C\,\frac {{\cal R}^{\lambda^+}_{x} {\cal R}^{\lambda^--2}_{y}\,|x'|^{\mu-2}}
{(t- s)^{\frac n2}\,|y'|^{\mu}\,r_y^{\varepsilon}} \,
 \frac {\delta} {t- s}\, \exp \left(-\frac {\sigma|x-y|^2}{t- s}\right);
\end{eqnarray*}
\begin{eqnarray*}
&&\left|{\cal G}_{ij}(x,y;t, s)-{\cal G}_{ij}(x,y;t, s^0)\right|
\le\int\limits_{s^0}^s|\partial_\tau {\cal G}_{ij}(x,y;t,\tau)|\,d\tau\\
&&\le C\,\frac {{\cal R}^{\lambda^+-2}_{x} {\cal R}^{\lambda^--2}_{y}\,|x'|^{\mu}}
{(t- s)^{\frac {n+2}2}\,|y'|^{\mu}\,r_x^{\varepsilon}r_y^{\varepsilon}} \,
 \frac {\delta} {t- s}\, \exp \left(-\frac {\sigma|x-y|^2}{t- s}\right).
\end{eqnarray*}
On the other hand, estimates (\ref{May0}) and (\ref{May3}) with
$|\alpha|=2$, $|\beta|=0$ gives
\begin{equation*}
\left|{\cal G}(x,y;t, s)-{\cal G}(x,y;t, s^0)\right| \le C\,\frac
{{\cal R}^{\lambda^+}_{x}{\cal R}^{\lambda^-}_{y}\,|x'|^{\mu-2}}
{(t-s)^{\frac n2}\,|y'|^{\mu}} \, \exp\left(-\frac{\sigma|x-y|^2}{t-s} \right).
\end{equation*}
\begin{equation*}
\left|{\cal G}_{ij}(x,y;t, s)-{\cal G}_{ij}(x,y;t, s^0)\right| \le
C\,\frac {{\cal R}^{\lambda^+-2}_{x}{\cal R}^{\lambda^-}_{y}\,|x'|^{\mu}}
{(t-s)^{\frac {n+2}2}\,|y'|^{\mu}r_x^{\varepsilon}} \,
\exp\left(-\frac{\sigma|x-y|^2}{t- s} \right).
\end{equation*}
Combination of these estimates gives
\begin{equation*}
\left|{\cal G}(x,y;t, s)-{\cal G}(x,y;t, s^0)\right| \le C\,\frac
{\delta^{\varepsilon}\,{\cal R}^{\lambda^+}_{x}{\cal R}^{\lambda^--2\varepsilon}_{y}
\,|x'|^{\mu-2}} {(t-s)^{\frac n2+\varepsilon}\,|y'|^{\mu}\,r_y^{\varepsilon^2}} \,
\exp\left(-\frac{\sigma|x-y|^2}{t- s} \right).
\end{equation*}
\begin{equation*}
\left|{\cal G}_{ij}(x,y;t, s)-{\cal G}_{ij}(x,y;t, s^0)\right| \le
\frac {C\,\delta^{\varepsilon}\,{\cal R}^{\lambda^+-2}_{x}{\cal
R}^{\lambda^--2\varepsilon}_{y}\,|x'|^{\mu}} {(t-s)^{\frac
{n+2}2+\varepsilon}\,|y'|^{\mu}r_x^{\varepsilon}r_y^{\varepsilon^2}}
\, \exp\left(-\frac{\sigma|x-y|^2}{t- s} \right).
\end{equation*}

Thus, the kernels in (\ref{difference1}) and (\ref{difference2})
satisfy the assumptions of Lemma \ref{L_p_1}, respectively, with
$$\varkappa=\varepsilon,\quad r=2,\quad \varepsilon_1=0,\quad \varepsilon_2=\varepsilon^2,\quad
\lambda_1=\lambda^+-2,\quad \lambda_2=\lambda^--2\varepsilon,
$$
and
$$\varkappa=\varepsilon,\quad r=0,\quad \varepsilon_1=\varepsilon,\quad \varepsilon_2=\varepsilon^2,\quad
\lambda_1=\lambda^+-2,\quad \lambda_2=\lambda^--2\varepsilon.
$$
This completes the proof.
\end{proof}

\begin{sats}\label{Lpq}
Let $1<p, q<\infty$, and let $\mu$ be subject to (\ref{mu}).
Suppose also that $|x'|^\mu f\in L_{p,q}({\cal K}\times\mathbb R)$. Then the function {\rm (\ref{Sept12a})}
solves the problem {\rm (\ref{Jan1a})} and satisfies the estimate {\rm (\ref{beztilde})}.
\end{sats}

\begin{proof}
As in Theorem \ref{tildeLpq}, it suffices to establish the estimate (\ref{beztilde}).

The estimate (\ref{tilde}) for $q=p$ provides boundedness of the
operators ${\cal G}$ and ${\cal G}_{ij}$ in $L_p(\mathbb
R^n\times \mathbb R)$, $1<p<\infty$, which gives the first condition
in \cite[Theorem 3.8]{BIN}. Lemma \ref{weak_2} is equivalent to the
second condition in this theorem. Therefore, Theorem 3.8 \cite{BIN}
ensures that these operators are bounded in $L_{p,q}(\mathbb R^n\times \mathbb R)$ for any $q\in\,(1,p)$.
For $q\in\,(p,\infty)$ this statement follows by duality arguments. This implies the
estimates of two last terms in (\ref{beztilde}). The estimate of the
first term follows now from (\ref{Jan1}).
\end{proof}

\begin{sats}\label{unique}
Under assumptions of Theorem \ref{Lpq} (respectively, Theorem \ref{tildeLpq}) the function {\rm (\ref{Sept12a})}
is a unique solution of the problem {\rm (\ref{Jan1a})} in the space $W^{2,1}_{p,q,(\mu)}$ (respectively,
$\widetilde {W}^{2,1}_{p,q,(\mu)}$).
\end{sats}

\noindent {\it Proof}. Let $\zeta=\zeta(\tau)$ be a smooth function on $\mathbb R_+$, which is
equal to $1$ for $\tau<1/2$ and $0$ for $\tau>3/4$.
We put $\chi_R(x;t)=\zeta(|x-y|/R)\zeta(\sqrt{|t-s|}/R)$. Then equation
(\ref{Jan1a}) can be written as
\begin{equation}\label{Jan1a'}
{\cal L}(\chi_Ru)=f_R
\end{equation}
where
$$
f_R=\chi_Rf+u\partial_t\chi_R-uA_{ij}D_iD_j\chi_R-2A_{ij}D_iu\,D_j\chi_R.
$$
Multiplying  (\ref{Jan1a'}) by $\Gamma_{\cal K}(y,x;s,t)$ and integrating over
${\cal K}\times(-\infty,s)$, we obtain
\begin{equation}\label{int}
u(y;s)=\int\limits_{-\infty}^s\int\limits_{\cal K}f_R(x;t)\Gamma_{\cal K}(y,x;s,t)\,dxdt.
\end{equation}
Since $|x'|^{\mu}f$, $|x'|^{\mu-2}u$ and, by Lemma \ref{gradient}, $|x'|^{\mu-1}D_ju$ belong to
$L_{p,q}({\cal K}\times\mathbb R)$ (respectively, to  $\widetilde L_{p,q}({\cal K}\times\mathbb R)$), we have
$|x'|^{\mu}f_R\to |x'|^{\mu}f$ in $L_{p,q}({\cal K}\times\mathbb R)$ (respectively, in
$\widetilde L_{p,q}({\cal K}\times\mathbb R)$) as $R\to\infty$. Using Theorem \ref{Lpq} (respectively,
Theorem \ref{tildeLpq}) we can pass to the limit in (\ref{int}) as $R\to\infty$ and obtain
$$u(y;s)=\int\limits_{-\infty}^s\int\limits_{\cal K}\Gamma_{\cal K}(y,x;s,t)f(x;t)\,dxdt.\eqno\square
$$

\begin{Rem}
To deal with initial-boundary value problem {\rm (\ref{init})}, one can extend
$\Gamma_{\cal K}(x,y;t,s)$ and $f(s)$ by zero for $s<0$.
In this case all statements of this Section hold true if we define $\lambda^+_c$ and $\lambda^-_c$
as described in Remark \ref{init1} (taking into account Remark \ref{init2}).
\end{Rem}

\section{Solvability of linear and quasilinear Dirichlet problems}\label{solv}

Let $\Omega$ be a bounded domain in $\mathbb R^n$.
We assume that there exists an $(n-m)$-dimensional submanifold without boundary (``edge'')
${\cal M}\subset\partial\Omega$ satisfying the following condition: for every
point $x_0\in{\cal M}$ there exists a neighborhood $U(x_0)$ and a diffeomorphism
$\Psi_{(x_0)}\,:\,U(x_0)\to B_{ R_0}^n$ ($R_0\le1$ is a constant independent of $x_0$)
such that
\begin{enumerate}
\item $\Psi_{(x_0)}(U(x_0)\cap\Omega)=B_{ R_0}^n\cap {\cal K}$\\
(note that  ${\cal K}={\cal K}(x_0)$ may depend on $x_0$);

\item $\Psi_{(x_0)}(x_0)=0$;

\item $\Psi_{(x_0)}(U(x_0)\cap\partial\Omega)=B_{ R_0}^n\cap \partial{\cal K}(x_0)$;

\item $\Psi'_{(x_0)}(x_0)={\cal I}_n$;

\item the norms of the Jacobi matrices $\Psi'_{(x_0)}(x)$ and $(\Psi_{(x_0)})'(\Psi_{(x_0)}(x))$ are uniformly bounded with respect to $x_0\in{\cal M}$ and to $x\in U(x_0)$.

\end{enumerate}

\begin{Rem}
If $m=n$ then ${\cal M}$ is not an edge but a conical point. In this case the wedge ${\cal K}$ degenerates
into the cone $K$. We do not exclude this case though it requires some trivial changes in notation which will not be mentioned in what follows.
\end{Rem}

We introduce two scales of functional spaces: ${\mathbb L}_{p,q,(\mu)}(Q)$ and
$\widetilde{\mathbb L}_{p,q,(\mu)}(Q)$,
with norms
$$
{\pmb \|}f{\pmb \|}_{p,q,(\mu),Q}=
\Vert (\widehat d(x))^{\mu}f\Vert_{p,q,Q}=
\Big (\int\limits_0^T\Big (\int\limits_\Omega
(\widehat d(x))^{\mu p}|f(x;t)|^pdx\Big )^{q/p}dt\Big )^{1/q}
$$
and
$$
\pmb {|\!|\!|}f\pmb {|\!|\!|}_{p,q,(\mu),Q}=
|\!|\!| (\widehat d(x))^{\mu}f|\!|\!|_{p,q,Q}=
\Big (\int\limits_\Omega\Big
(\int\limits_0^T (\widehat d(x))^{\mu q}|f(x;t)|^qdt\Big )^{p/q}dx\Big)^{1/p}
$$
respectively, where $\widehat d(x)$ stands for the distance from
$x \in \Omega$ to ${\cal M}$. For $p=q$ these spaces coincide, and we write
${\mathbb L}_{p,(\mu)}(Q)$.

We denote by ${\mathbb W}^{2,1}_{p,q,(\mu)} (Q)$ and $\widetilde
{\mathbb W}^{2,1}_{p,q,(\mu)} (Q)$ the set of functions with the
finite norms
$$
\|u \|_{{\mathbb W}^{2,1}_{p,q,(\mu)} (Q)}={\pmb \|}\partial_tu {\pmb \|}_{p,q,(\mu),Q}
+{\pmb \|}D(Du) {\pmb \|} _{p,q, (\mu),Q}+{\pmb \|}u {\pmb \|} _{p,q, (\mu-2),Q}
$$
and
$$
\|u \|_{\widetilde{\mathbb W}^{2,1}_{p,q,(\mu)} (Q)}=\pmb{|\!|\!|} \partial_tu \pmb{|\!|\!|}_{p,q, (\mu),Q}
+\pmb{|\!|\!|} D(Du) \pmb{|\!|\!|}_{p,q, (\mu),Q}+\pmb{|\!|\!|} u \pmb{|\!|\!|}_{p,q, (\mu-2),Q}
$$
respectively, satisfying boundary condition $u|_{\partial'Q}=0$.

We say $\partial\Omega \in {\cal W}^2_{p,(\mu)}$ if $\partial\Omega\setminus{\cal M} \in W^2_{p,loc}$
and for any point $x_0\in{\cal M}$ the mappings $\widetilde\Psi_{(x_0)}(x)=\Psi_{(x_0)}(x)-(x-x_0)$ have
finite norms
$$\| (\widehat d(x))^{\mu}D(Du)\widetilde\Psi_{(x_0)}\|_{p,U(x_0)}+
\| (\widehat d(x))^{\mu-2}\widetilde\Psi_{(x_0)}\|_{p,U(x_0)};
$$
these norms are uniformly bounded with respect to $x_0$.

We set $\widehat\mu(p,q)=1-\frac{n}{p}-\frac{2}{q}$.

\subsection{Linear Dirichlet problem in bounded domains}\label{Sect5.1}

We consider the initial-boundary value problem
\begin{equation}\label{DesS4}
{\mathfrak L}u\equiv \partial_tu-a_{ij}(x;t)D_iD_ju+b_i(x;t)D_iu=f(x;t)
\ \ {\textup{in}} \ \ Q,\quad\ u|_{\partial'Q}=0,
\end{equation}
where the leading coefficients $a_{ij}\in {\cal C}(\overline{\Omega}\to L_\infty(0,T))$
satisfy assumptions $a_{ij}=a_{ji}$ and
$$\nu|\xi|^2\le a_{ij}(x;t)\xi_i\xi_j\le \nu^{-1}|\xi|^2, \qquad
\xi\in{\mathbb R}^n, \quad \nu=const>0.
$$

We denote by ${\cal L}_{x_0}$ the operator of the form (\ref{Jan1}) with frozen coefficients
$A_{ij}(t)=a_{ij}(x_0;t)$ and define the quantities
\begin{equation}\label{lambda-c}
\widehat\lambda_c^{\pm}=\inf\limits_{x_0\in{\cal M}}\lambda_c^{\pm}({\cal K}(x_0),{\cal L}_{x_0}).
\end{equation}

\begin{sats}\label{linear}
Let $1<p,q<\infty$ and $2-\frac mp-\widehat\lambda_c^+<\mu<m-\frac mp+\widehat\lambda_c^-$.
\medskip

{\bf 1}. Let $ b_i \in {\mathbb L}_{{\overline{p}},{\overline{q}},(\overline{\mu})}(Q)$ where
$\displaystyle {\overline{p}}$ and $\displaystyle {\overline{q}}$
are subject to
$$\overline{p}\ge p;\quad \left[
\begin{array}{ll}
\overline{q}=q;&
\widehat\mu(\overline{p},\overline{q})>0\\
q<\overline{q}<\infty;&
\widehat\mu(\overline{p},\overline{q})=0
\end{array}\right.,
$$
while $\overline {\mu}$ satisfies
\begin{equation}\label{mumu}
{\overline{\mu}}=\min\{\mu, \max\{\widehat\mu(p,q),0\}\}.
\end{equation}

Suppose also that $\partial\Omega \in {\cal W}^2_{\overline{p},(\overline{\mu})}\cap W^1_{\infty}$.
Then, for any $f \in {\mathbb L}_{p,q,(\mu)}(Q)$, the initial-boundary value problem {\rm (\ref{DesS4})}
has a unique solution $u\in {\mathbb W}^{2,1}_{p,q,(\mu)}(Q)$. Moreover, this solution satisfies
$$\|u \|_{{\mathbb W}^{2,1}_{p,q,(\mu)} (Q)} \le C {\pmb\|}f{\pmb \|}_{p,q,(\mu)},
$$
where $C$ does not depend on $f$.\medskip

{\bf 2}. Let $ b_i \in \widetilde{\mathbb L}_{{\overline{p}},{\overline{q}},(\overline{\mu})}(Q)$ where
$\displaystyle {\overline{p}}$ and $\displaystyle {\overline{q}}$ are subject to
$$\overline{q}\ge q;\quad \left[
\begin{array}{ll}
\overline{p}=p;&
\widehat\mu(\overline{p},\overline{q})>0\\
p<\overline{p}<\infty;&
\widehat\mu(\overline{p},\overline{q})=0
\end{array}\right.,
$$
while $\overline {\mu}$ satisfies (\ref{mumu}). Suppose also that $\partial\Omega$ satisfies the same
condition as in the part {\bf 1}. Then, for any $f \in \widetilde{\mathbb L}_{p,q,(\mu)}(Q)$, the
initial-boundary value problem {\rm (\ref{DesS4})} has a unique solution
$u\in\widetilde{\mathbb W}^{2,1}_{p,q,(\mu)}(Q)$. Moreover, this solution satisfies
$$\|u \|_{\widetilde{\mathbb W}^{2,1}_{p,q,(\mu)} (Q)} \le C \pmb{|\!|\!|} f \pmb{|\!|\!|}_{p,q,(\mu)},
$$
where $C$ does not depend on $f$.

\end{sats}

\begin{Rem}
 These assertions generalize {\rm \cite[Theorem 2.1']{Na1}}.
\end{Rem}

\begin{proof}
The standard scheme, see, e.g., \cite[Ch.IV, Sect. 9]{LSU},
including partition of unity, local rectifying of $\partial\Omega$
and coefficients freezing, reduces the proof to the coercive
estimates for the model problems to equation (\ref{Jan1}) in the
whole space, in the half-space and in the wedge ${\cal K}$. The
first two estimates were obtained in \cite[Theorem 1.1]{Kr} and
\cite[Theorems 1 and 4]{KN}. The last one is established in our
Theorems \ref{tildeLpq} and \ref{Lpq}. By the H\"older inequality
and the embedding theorems (see, e.g., \cite[Theorems 10.1 and
10.4]{BIN}), the assumptions on $b_i$ guarantee that the lower-order
terms in (\ref{DesS4}) belong to desired weighted spaces ${\mathbb
L}_{p,q,(\mu)}(Q)$ and $\widetilde{\mathbb L}_{p,q,(\mu)}(Q)$
respectively. By the same reasons, the requirements on
$\partial\Omega$ ensure the invariance of assumptions on $b_i$ under
rectifying of $\partial\Omega$ in the neighborhood of a smooth point
and under diffeomorphism $\Psi_{(x_0)}$ in the neighborhood of
$x_0\in{\cal M}$.
\end{proof}

\subsection{Quasilinear Dirichlet problem in bounded domains}\label{Sect5.2}

In this subsection we suppose, in addition to the assumptions 1--5 from the beginning of Section
\ref{solv}, that for any $x_0 \in {\cal M}$ corresponding wedge ${\cal K}(x_0)$ is acute that is
$K(x_0) \subset K^{\theta}$ where $\theta < \frac{\pi}{2}$ does not depend on $x_0$.\medskip

We consider the initial-boundary value problem
\begin{equation}\label{quasi}
\partial_tu-a_{ij}(x;t;u;Du)D_iD_ju + a(x;t;u;Du) = 0 \quad\mbox{in}\ \ Q,\qquad
u|_{\partial'Q}=0.
\end{equation}
We suppose that the first derivatives of the coefficients
$a_{ij}(x;t;z;\mathfrak p)$ with respect to $x$, $z$ and $\mathfrak
p$ are locally bounded and the following inequalities hold for all
$(x;t) \in Q$, $z \in {\mathbb R}^1$ and $\mathfrak p \in{\mathbb
R}^n$ with some positive $\nu$ and $\nu_1$:
\begin{equation}\label{quasi1}
\gathered \nu |\xi|^2  \leq a_{ij}(x;t;z;\mathfrak p)\xi_i
\xi_j \leq
\nu^{-1}|\xi|^2\qquad  \forall \xi \in {\mathbb R}^n, \\
|a(x;t;z;\mathfrak p)| \leq \nu_1 |\mathfrak p|^2+
b(x;t)|\mathfrak p|+\Phi(x;t),\\
\left|\frac {\partial a_{ij}(x;t;z;\mathfrak p)}{\partial
\mathfrak p}
\right|\leq \frac {\nu_1}{1+|\mathfrak p|},\\
\left|{\mathfrak p}\cdot \frac {\partial a_{ij}(x;t;z;\mathfrak
p)}{\partial z}+ \frac {\partial a_{ij}(x;t;z;\mathfrak
p)}{\partial x}\right|\le \nu_1|\mathfrak p|+\Phi_1(x;t).
\endgathered
\end{equation}

Similarly to the previous subsection, we denote by ${\cal L}_{x_0}$ the operator of the form (\ref{Jan1})
with frozen coefficients $A_{ij}(t)=a_{ij}(x_0;t;0;0)$ and define the quantity $\widehat\lambda_c^{\pm}$ by the
formula (\ref{lambda-c}).

\begin{sats}\label{quasilinear}
{\bf 1}. Let the following assumptions be satisfied:
\begin{description}
\item[{(i)}] $1<q\leq p<\infty$, \quad $\widehat{\mu}(p,q)>\max\{2-\frac{m}{p}-\widehat\lambda_c^+,0\}$;
\item[{(ii)}] $2-\frac{m}{p}-\widehat\lambda_c^+<\mu<\widehat{\mu}(p,q)$, \quad
$\partial\Omega \in {\cal W}^2_{p,(\mu)}$;
\item[{(iii)}] functions $a_{ij}$ and $a$ satisfy the structure conditions {\rm (\ref{quasi1})};
\item[{(iv)}] $b,\Phi\in{\mathbb L}_{p,q,(\mu)}(Q)$;
\item[{(v)}] $\Phi_1 \in {\mathbb L}_{p_1,q_1,(\mu_1)}(Q)$,\quad $q_1\leq
p_1<\infty$,\ \  $\widehat{\mu}(p_1,q_1)>\max\{\mu_1,0\}$;
\item[{(vi)}] $a(\cdot;z;\mathfrak p)$ is continuous w.r.t.
$(z,\mathfrak p)$ in the norm ${\pmb \|}\cdot{\pmb \|}_{p,q,(\mu),Q}$.
\end{description}
Then the problem {\rm (\ref{quasi})} has a solution $u\in {\mathbb W}^{2,1}_{p,q,(\mu)}(Q)$.
\medskip

{\bf 2}. Let the following assumptions be satisfied:
\begin{description}
\item[{(i)}] $1<p\leq q<\infty$, \quad $\widehat{\mu}(p,q)>\max\{2-\frac{m}{p}-\widehat\lambda_c^+,0\}$;
\item[{(ii)}] $2-\frac{m}{p}-\widehat\lambda_c^+<\mu<\widehat{\mu}(p,q)$, \quad
$\partial\Omega \in {\cal W}^2_{p,(\mu)}$;
\item[{(iii)}] functions $a_{ij}$ and $a$ satisfy the structure conditions {\rm (\ref{quasi1})};
\item[{(iv)}] $b,\Phi\in\widetilde{\mathbb L}_{p,q,(\mu)}(Q)$;
\item[{(v)}] $\Phi_1 \in \widetilde{\mathbb L}_{p_1,q_1,(\mu_1)}(Q)$,\quad
$p_1\leq q_1<\infty$,\ \
$\widehat{\mu}(p_1,q_1)>\max\{\mu_1,0\}$;
\item[{(vi)}] $a(\cdot;z;\mathfrak p)$ is continuous w.r.t. $(z,\mathfrak p)$ in
the norm $\pmb{|\!|\!|}\cdot\pmb{|\!|\!|}_{p,q,(\mu),Q}$.
\end{description}
Then the problem {\rm (\ref{quasi})} has a solution $u\in\widetilde{\mathbb W}^{2,1}_{p,q,(\mu)}(Q)$.
\end{sats}

\begin{proof}
The proof by the Leray--Schauder principle is also rather
standard, see, \cite[Ch.V, Sect. 6]{LSU}. In the case when the leading
coefficients are continuous in $t$, these assertions were proved
in \cite[Theorem 1.1]{Na1}.

Note that by Theorem \ref{lambda}, statement {\bf 4}, the assumption $\theta < \pi /2$ implies
$\widehat\lambda_c^+ >1$. This guarantees that for sufficiently large $p$ and $q$
the interval for $\mu$ in {\bf(ii)} is non-empty. Further, since $\widehat{\mu}(p,q)>0$,
the assumption {\bf(ii)} implies $\partial\Omega\setminus{\cal M}\in{\cal C}^1_{loc}$. The solvability
of the corresponding linear problem follows from Theorem \ref{linear} while
required a priori estimates obtained in \cite{Na1}, see also \cite{LU1} and \cite{AN},
do not use continuity of $a_{ij}$ with respect to $t$.
\end{proof}

Note that in Theorem \ref{quasilinear} for $p>q$ we deal with ${\mathbb L}_{p,q,(\mu)}(Q)$ scale
while for $p<q$ we deal with $\widetilde{\mathbb L}_{p,q,(\mu)}(Q)$ scale. The reason is that all
a priori estimates for quasilinear equations are based on the Aleksandrov--Krylov
maximum principle. Up to now this statement is proved only if the right-hand side
of the equation belongs to the space with stronger norm, see \cite{N87} and \cite{Na1}.

\section{Appendix}\label{append}

\begin{lem}\label{int_est} Let
$\mu^2<\nu^2 \Big (\Lambda_{\cal D}+\frac {(m-2)^2}{4}\Big )$
and let $0<\kappa_1<\kappa_2\leq 1$. Then the estimate
\begin{multline}\label{May9c}
\int\limits_{{\cal Q}_{\kappa_1R}^K(0;t_0)}|x'|^{2\mu}|Du|^2dxdt+
\int\limits_{{\cal Q}_{\kappa_1R}^K(0;t_0)}|x'|^{2\mu-2}|u|^2dxdt\\
\leq CR^{2\mu-2}\int\limits_{{\cal Q}_{\kappa_2R}^K(0;t_0)}
|u|^2dxdt
\end{multline}
holds for any function $u\in {\cal V}({\cal Q}_R^K(0;t_0))$
satisfying (\ref{0.5}). The constant $C$ may depends on $\mu$,
$\kappa_1$ and $\kappa_2$.
\end{lem}

\begin{proof}
Without loss of generality, we assume $t_0=0$. Further, since the estimate
(\ref{May9c}) depends only on $K$ and $\nu$, we can suppose $R=1$;
general case can be reduced to this one by dilation with respect to $x$ and $\sqrt{|t|}$.

We put $\rho_\varepsilon=(|x'|^2+\varepsilon)^{1/2}$, where
$\varepsilon$ is a small positive number. Let us fix two real
numbers $R_1$ and $R$ such that $0<R_1<R_2\leq 1$. We choose two
smooth cut-off functions: $\chi=\chi(x)$ supported in ${\cal B}_{R_2}^K$ and
equal to $1$ on ${\cal B}_{R_1}^K$ and $\zeta=\zeta(t)$ which is equal to
$1$ for $t>-R_1^2$ and to $0$ for $t<-R_2^2$. Then we have
\begin{eqnarray}\label{May9a}
&&\int\limits_{{\cal Q}_{R_2}^K} {\cal L}u\; \rho_\varepsilon^{2\mu}\zeta^2\chi^2
udx=\frac{1}{2}\int\limits_{{\cal B}_{R_2}^K}u^2(x;0)
\chi^2\rho_\varepsilon^{2\mu}|u|^2dx\nonumber\\
&&-\int\limits_{{\cal Q}_{R_2}^K}\zeta'\zeta\chi^2u^2dxdt
+\int\limits_{{\cal Q}_{R_2}^K}\zeta^2\chi^2A_{ij}\partial_{x_i}(\rho_\varepsilon^{-\mu}
v)\partial_{x_j}(\rho_\varepsilon^{\mu}v)dxdt\nonumber\\
&&+\int\limits_{{\cal Q}_{R_2}^K}\zeta^2A_{ij}\partial_{x_i}(\rho_\varepsilon^{-\mu}
v)\rho_\varepsilon^{\mu}v\partial_{x_j}\chi^2dxdt,
\end{eqnarray}
where $v=\rho_\varepsilon^\mu u$. Direct calculations together
with (\ref{0.3}) give
\begin{eqnarray}\label{Uk2}
&&\int\limits_{{\cal B}_{R_2}^K}\chi^2A_{ij}\partial_{x_i}(\rho_\varepsilon^{-\mu}
v)\partial_{x_j}(\rho_\varepsilon^{\mu}v)dx=\int\limits_{{\cal B}_{R_2}^K}\chi^2A_{ij}\partial_{x_i}
v\;\partial_{x_j}v\,dx\\
&&-\mu^2\int\limits_{{\cal B}_{R_2}^K}\rho_\varepsilon^{-4}\chi^2A_{ij}x_ix_j|v|^2\,dx\geq
\nu^{-1} \int\limits_{{\cal
K}}\chi^2(\nu^2|Dv|^2-\mu^2\rho_\varepsilon^{-2}|v|^2)\,dx.\nonumber
\end{eqnarray}
Using the Hardy inequality
$$
\int\limits_K|D_{x'}f|^2dx'\geq \Big (\frac{(m-2)^2}{4}+\Lambda_{\cal D}\Big
)\int\limits_K|x'|^{-2}|f|^2dx'
$$
and the identity
$$
\int\limits_{{\cal B}_{R_2}^K}\chi^2 |Dv|^2dx=\int\limits_{{\cal B}_{R_2}^K}\Big (|D(\chi v)|^2+\chi\Delta\chi\,v^2 \Big )dx,
$$
we obtain
\begin{equation}\label{Uk3}
\int\limits_{{\cal B}_{R_2}^K}\chi^2|Dv|^2dx\geq \Big (\frac{(m-2)^2}{4}+\Lambda_{\cal D}\Big
)\int\limits_{{\cal B}_{R_2}^K}|x'|^{-2}\chi^2v^2dx-c\int\limits_{{\cal B}_{R_2}^K}v^2dx.
\end{equation}
 Since ${\cal L}u=0$ in ${\cal Q}_1^K$ then using (\ref{Uk2}) and (\ref{Uk3}), we
derive from (\ref{May9a}) the following estimate
\begin{eqnarray*}
&&0\geq \frac{1}{2}\int\limits_{{\cal B}_{R_2}^K}u^2(x;0)
\chi^2\rho_\varepsilon^{2\mu}|u|^2dx-\!\int\limits_{{\cal Q}_{R_2}^K}\!\rho_\varepsilon^{2\mu}|u|^2dxdt
+\delta\!\int\limits_{{\cal Q}_{R_2}^K}\!\chi^2|Du|^2dxdt\\
&&+c_\mu\int\limits_{{\cal Q}_{R_2}^K}\zeta^2\chi^2 v^2\rho_\varepsilon^{-2}dxdt
-\!\int\limits_{{\cal Q}_{R_2}^K}\!\zeta^2A_{ij}u^2
\partial_{x_i}\big (\rho_\varepsilon^{2\mu}\chi\partial_{x_j}\chi(x)\big
)dxdt\,,
\end{eqnarray*}
where $c_\mu\equiv(\nu -\delta)\Big (\frac{(m-2)^2}{4}+\Lambda_{\cal D}\Big
)-\nu^{-1}\mu^2$ is positive if $\delta>0$ is sufficiently small.
 Here we performed a partial integration in
the last integral in (\ref{May9a}). Taking the limit in the last
inequality as $\varepsilon\to 0$, we get
\begin{equation}\label{May9e}
\int\limits_{{\cal Q}_{R_1}^K}\big (|x'|^{2\mu}|Du|^2+ |x'|^{2\mu-2}|u|^2\big )dxdt\leq
C\int\limits_{{\cal Q}_{R_2}^K}|x'|^{2\mu}|u|^2dxdt,
\end{equation}
where $C$ depends on $\mu$. If $\mu$ is positive then taking $R_1=\kappa_1$ and
 $R_2=\kappa_2$ in (\ref{May9e}), we obtain
\begin{equation}\label{May9ct}
\int\limits_{{\cal Q}_{\kappa_1}^K}|x'|^{2\mu-2}|u|^2dxdt\leq
C(\mu)\int\limits_{{\cal Q}_{\kappa_2}^K} |u|^2dxdt.
\end{equation}
Otherwise, iterations of (\ref{May9e}), without the first term in the
left-hand side, again lead to (\ref{May9ct}). The estimate of the first
term in the left-hand side of (\ref{May9c}) follows from
(\ref{May9e}) and (\ref{May9ct}).
\end{proof}

\begin{lem}\label{gradient}
Let $1<p,q<\infty$ and $\mu\in\mathbb R$.
For any $u\in W^{2,1}_{p,q,(\mu)}$ (respectively, $u\in\widetilde {W}^{2,1}_{p,q,(\mu)}$) the following
estimate holds:
\begin{equation}\label{estgrad}
\||x'|^{\mu-1} Du\|_{p,q} \le C \|u \|_{W^{2,1}_{p,q,(\mu)}};\qquad
|\!|\!||x'|^{\mu-1} Du|\!|\!|_{p,q}\le C\|u \|_{\widetilde{W}^{2,1}_{p,q,(\mu)}}.
\end{equation}
\end{lem}

\begin{proof}
For any $\xi'' \in{\mathbb R}^{n-m}$ and $\rho>0$ the inequality
$$\int\limits_{\widetilde{\cal P}_{\rho,2}(\xi'')}
\rho^{-p}|Du|^p\, dx\le C\int\limits_{\widetilde{\cal P}_{2\rho,8}(\xi'')}
\left(\rho^{-2p}|u|^p+|D(Du)|^p\right)\, dx
$$
follows from standard embedding theorem (the set $\widetilde{\cal P}_{\rho,\gamma}(\xi'')$ is introduced in the proof of Theorem \ref{tildeLpq}).

Using a proper partition of unity in ${\cal K}$, we arrive at
$$\int\limits_{\cal K}|Du|^p|x'|^{\mu p-p} dx\le
C\int\limits_{\cal K}\left(|u|^p|x'|^{\mu p-2p} +|D(Du)|^p|x'|^{\mu p}\right)\, dx.
$$
This implies the first estimate in (\ref{estgrad}).\medskip

In a similar way, the embedding
\begin{multline*}
\int\limits_{\widetilde{\cal P}_{\rho,2}(\xi'')}\bigg( \int\limits_
{\mathbb R}\rho^{-q}|Du|^q \, dt\bigg)^{p/q} dx\\
\le C\int\limits_{\widetilde{\cal P}_{2\rho,8}(\xi'')}
\bigg( \int\limits_{\mathbb R}\left(|\partial_t u|^q+|D(Du)|^q+\rho^{-2q}|u|^q\right)\, dt\bigg)^{p/q} dx
\end{multline*}
implies the second estimate in (\ref{estgrad}).
\end{proof}

\begin{lem}\label{aux}
Let $\alpha>-1$, $\alpha+\beta>-m$,
$$F(z',w')\equiv\exp(-\sigma|z'-w'|^2)d(z')^{\alpha}|z'|^{\beta}(|z'|+1)^{\gamma}.
$$
Then
$${\cal I}\equiv\int\limits_{K}F(z',w')\,dz'\le
 C\,(|w'|+1)^{\alpha+\beta+\gamma},$$
where $C$ may depend on $\sigma$, $\alpha$, $\beta$, $\gamma$ and
$K$.
\end{lem}

\begin{proof}
First, let $|w'|\le 2$. Then $|z'-w'|\ge \max\{|z'|-2,0\}$, and
$${\cal I}\le C\int\limits_{\omega}r_z^{\alpha}(\Theta)\,d\Theta
\int\limits_0^{\infty}\exp(-\sigma\cdot\max\{r-2,0\}^2)r^{\alpha+\beta+m-1}(r+1)^{\gamma}\,dr.
$$
We recall that $r_z(\Theta)=\frac{d(z')}{|z'|}\asymp\,\mbox{dist}(\Theta,\partial\omega)$.
Since $\alpha>-1$ and $\alpha+\beta>-m$, both integrals converge,
and we obtain ${\cal I}(w')\le C$.\medskip

Now let $|w'|\equiv\rho\ge 2$. Then we split ${\cal I}=I_1+I_2$,
where
$$ I_1=\int\limits_{K\cap B_{\frac{\rho}2}^m(w')}F(z',w')\,dz',\qquad
 I_2=\int\limits_{K\setminus B_{\frac{\rho}2}^m(w')}F(z',w')\,dz'.
$$

Since $|z'|\asymp \rho\asymp \rho+1$ in $B_{\frac{\rho}2}^m(w')$, we
obtain
$$I_1\le C\rho^{\alpha+\beta+\gamma}\int\limits_{K\cap B_{\frac{\rho}2}^m(w')}
\exp(-\sigma|z'-w'|^2)\ r_z^{\alpha}(\Theta)\,dz'.
$$
We split $B_{\lceil \frac{\rho}2\rceil}^m(w')$ into layers
$B_N^m(w')\setminus B_{N-1}^m(w')$, $N=1,\dots,\lceil
\frac{\rho}2\rceil$, and estimate as follows:
$$I_1\le C\rho^{\alpha+\beta+\gamma}\sum\limits_{N=1}^{\lceil \frac{\rho}2\rceil}
\exp(-\sigma(N-1)^2)N^n\cdot\int\limits_{\omega}r_z^{\alpha}(\Theta)\,d\Theta \le C\rho^{\alpha+\beta+\gamma}.
$$

To estimate $I_2$, we note that if $|z'|> \frac{3\rho}2$ then
$|z'-w'|\ge\frac{|z'|}3\ge\frac{\rho}4+\frac{|z'|}6$. If, otherwise,
$|z'|\le \frac{3\rho}2$ then $|z'-w'|\ge\frac{\rho}2$ implies again
$|z'-w'|\ge\frac{\rho}4+\frac{|z'|}6$. Thus,
\begin{multline*}
I_1\le C\exp\left(-\frac {\sigma\rho^2}{16}\right)
\int\limits_{\omega}r_z^{\alpha}(\Theta)\,d\Theta\\
\times\int\limits_0^{\infty}\exp\left(-\frac{\sigma
r^2}{36}\right)r^{\alpha+\beta+m-1}(r+1)^{\gamma}\,dr \le
C\exp\left(-\frac {\sigma\rho^2}{16}\right),
\end{multline*}
and the statement follows.
\end{proof}

The next lemma is a generalization of \cite[Lemma 10]{KN}, see also
\cite[Lemma 3.2]{Na}.
\begin{lem}\label{L_p_1}
Let $1<p<\infty$, $\sigma>0$, $\varkappa>0$, $0\le r\le 2$,
$0\le\varepsilon_1<\frac 1p$, $0\le\varepsilon_2<1-\frac 1p$,
${\lambda_1}+{\lambda_2}>-m$, and let $\mu$ satisfy
\begin{equation}\label{mu_m}
-\frac mp-\lambda_1<\mu<m-\frac mp+\lambda_2.
\end{equation}
 Also let the kernel ${\cal T}(x,y;t,s)$ satisfy the inequality
\begin{equation}\label{kernel_k1}
|{\cal T}(x,y;t,s)|  \le C\,\frac {\delta^\varkappa\,{\cal
R}_{x}^{\lambda_1+r} {\cal R}_{y}^{\lambda_2}\,|x'|^{\mu-r}}
{(t-s)^{\frac
{n+2-r}2+\varkappa}\,|y'|^{\mu}\,r_x^{\varepsilon_1}r_y^{\varepsilon_2}}
 \exp\left(-\frac{\sigma|x-y|^2}{t-s} \right),
\end{equation}
for $x,y\in{\cal K}$ and $t>s+\delta$. Then for any $s^0>0$ the norm
of the operator
$${\cal T}\ :\ L_{p,1}({\cal K}\times\ (s^0-\delta,s^0+\delta))\ \to \
L_{p,1}({\cal K}\times\ (s^0+2\delta,\infty))$$ does not exceed a
constant independent of $\delta$ and $s^0$.
\end{lem}

\begin{proof}
Let $h\in L_{p, 1}$ be supported in the layer $|s-s^0|\le\delta$.
Using (\ref{kernel_k1}) and the H\"older inequality, we have
\begin{eqnarray}\label{1st}
|({\cal T}h)(x;t)|&\le& C\,\int\limits_{-\infty}^{t}\frac
{\delta^{\varkappa} ds} {(t-s)^{\varkappa+1-\frac r2}}\nonumber \\
&\times&\biggl(\ \int\limits_{\cal K} \exp
\left(-\frac{\sigma|x-y|^2} {t- s}\right)\frac {|x'|^{(\mu-r)
p}R^{(\lambda_1+r)p}_{x}\ |h(y; s)|^p} {r_x^{\varepsilon_1p}\,(t-
s)^{\frac n2}}\
dy\biggr)^{\frac{1}{p}}\nonumber\\
&\times&\biggl(\ \int\limits_{\cal K}\exp
\left(-\frac{\sigma|x-y|^2}{t- s} \right) \frac {{\cal
R}_y^{\lambda_2p'}}{|y'|^{\mu p'}r_y^{\varepsilon_2p'}\,(t-
s)^{\frac n2}}\ dy \biggr)^{\frac{1}{p'}}.
\end{eqnarray}
Denote by ${\cal I}_1$ the integral in the last large brackets.
Using the change of variable $x=w\sqrt{t- s}$, $y=z\sqrt{t- s}$,
and, in the case $m<n$, integrating with respect to $z''$, we obtain
$${\cal I}_1=C\int\limits_{K}\frac {\exp \left(-\sigma |z'-w'|^2 \right)
|z'|^{(\lambda_2-\mu+\varepsilon_2) p'}\, dz'} {(t-s)^{\mu
p'}d(z')^{\varepsilon_2p'}(|z'|+1)^{\lambda_2p'}}\leq C
\left(|x'|+\sqrt {t- s}\right)^{-\mu p'}\!.
$$
(the last inequality is by Lemma \ref{aux}).

From this estimate and (\ref{1st}), it follows that
\begin{eqnarray*}
&&\int\limits_{s^0+2\delta}^{\infty}\Vert ({\cal T}h)(\cdot;
t)\Vert_p\ dt\le C\,\int\limits_{s^0+2\delta}^{\infty} \biggl(\
\int\limits_ {\cal K}\biggl(\ \int\limits_{-\infty}^t \biggl(\
\int\limits_{\cal K}
\exp \left(-\frac {\sigma|x-y|^2}{t- s} \right)\\
&\times&\ \frac {|x'|^{(\lambda_1+\mu)p}|h(y; s)|^p\, dy}
{{\left(|x'|+\sqrt {t-
s}\right)^{(\lambda_1+\mu+r)p}}r_x^{\varepsilon_1p}\, ({t-
s})^{\frac n2}}\biggr)^{\frac{1}{p}}\frac {\delta^{\varkappa}\, ds}
{(t- s)^{\varkappa+1-\frac r2}}\biggr)^p dx\biggr)^{\frac{1}{p}}dt.
\end{eqnarray*}
Using Minkowski inequality, we estimate the right-hand side by
\begin{eqnarray*}
&&C\,\int\limits_{ s^0+2\delta}^{\infty}\int\limits_ { s^0-\delta}^
{ s^0+\delta}\ \frac {\delta^{\varkappa}\, ds dt}
{(t-s)^{\varkappa+1-\frac r2}}\,\biggl(\ \int\limits_{\cal K}
\int\limits_{\cal K}\exp \left(-\frac {\sigma|x-y|^2}{t- s} \right)\\
&\times&\, \frac {|x'|^{(\lambda_1+\mu)p}|h(y; s)|^p\, dy dx}
{\left(|x'|+\sqrt {t-
s}\right)^{(\lambda_1+\mu+r)p}r_x^{\varepsilon_1p}\,
({t-s})^{\frac n2}}\biggr)^{\frac{1}{p}}\\
&\le &C\,\int\limits_{s^0-\delta}^{ s^0+\delta} \| h(\cdot,s)\|_p\
ds \int\limits_ {s^0+2\delta}^{\infty}\frac {\delta^{\varkappa} dt}
{(t-s)^{\varkappa+1-\frac r2}} \cdot \sup\limits_y{\cal
I}_2^{\frac{1}{p}},
\end{eqnarray*}
where
$$
{\cal I}_2=\int\limits_{\cal K} \exp \left(-\frac
{\sigma|x-y|^2}{t-s} \right) \frac {|x'|^{(\lambda_1+\mu)p}\ dx}
{{\left(|x'|+\sqrt
{t-s}\right)^{(\lambda_1+\mu+r)p}}r_x^{\varepsilon_1p}\,
({t-s})^{\frac n2}}.
$$
We apply the change of variables $x=z\sqrt{t- s}$ and $y=w\sqrt{t-
s}$ and, in the case $m<n$, integrate with respect to $z''$. This
leads to
$${\cal I}_2= C\,\int\limits_{K}\frac
{\exp\left(-\sigma |z'-w'|^2 \right)
|z'|^{(\lambda_1+\mu+\varepsilon_1)p}\,dz'} {(t- s)^{\frac
{rp}2}d(z')^{\varepsilon_1p}(|z'|+1)^{(\lambda_1+\mu+r)p}}\le C\,(t-
s)^{-\frac {rp}2}
$$
(the last inequality is by Lemma \ref{aux}).

Thus,
$$\int\limits_{ s^0+2\delta}^{\infty}\Vert ({\cal K}h)(\cdot; t)\Vert_p\ dt
\le C\,\Vert h\Vert_{p,1}\, \sup\limits_{|s- s^0|<\delta}
\int\limits_ { s^0+2\delta}^{\infty}\frac {\delta^{\varkappa} dt}
{(t-s)^{1+\varkappa}}\le C\,\Vert h\Vert_{p,1},
$$
which completes the proof.
\end{proof}

The next proposition is proved in \cite[Lemmas A.1 and A.3 and Remark A.2]{KN}, see also
\cite[Lemmas 2.1 and 2.2]{Na}.

\begin{prop}\label{L_p}
Let $1\le p\le\infty$, $\sigma>0$, $0<r\le 2$, ${\lambda_1}+{\lambda_2}>-m$,
and let $\mu$ satisfy (\ref{mu_m}). Suppose also that the kernel ${\cal T}(x,y;t,s)$
satisfies the inequality
\begin{equation}\label{kernel_k}
|{\cal T}(x,y;t,s)|\le C\,\frac {R_{x}^{\lambda_1+r}
R_{y}^{\lambda_2}}{(t-s)^{\frac {n+2-r}2}}\,
\frac{|x'|^{\mu-r}}{|y'|^{\mu}}\,\exp
\left(-\frac{\sigma|x-y|^2}{t-s} \right),
\end{equation}
 for $t>s$. Then the integral operator ${\cal T}$ is bounded in
$L_p(\mathbb R^n\times\mathbb R)$ and in $\widetilde L_{p,\infty}(\mathbb R^n\times\mathbb R)$.
\end{prop}

\vspace{6mm}

\noindent {\bf Acknowledgements.} V.~K. was supported by the
Swedish Research Council (VR). A.~N. was supported by Federal Scientific and Innovation Program
and by RFBR grant 09-01-00729. He also acknowledges the Link\"oping University for the financial
support of his visits in November 2009 and in February 2011.

\end{document}